\documentclass{article}

\usepackage{authblk}

\usepackage{hyperref}
\usepackage{amssymb,amsfonts,amsmath,amsthm,pigpen,bbm,
,enumitem
,stmaryrd
,pifont
,hyphenat
,epigraph
,tikz
,etoolbox
,mparhack
,mathtools
,microtype
,lmodern
,datetime
,cite
,CJKutf8
,marvosym
,chessfss,
marvosym,wasysym,mathrsfs,afterpage,
tikz-cd,tikz}
\usepackage[all]{xypic}

\newcommand\blankpage{
    \null
    \thispagestyle{empty}
    \addtocounter{page}{-1}
    \newpage}

\newtheorem{thm}{Theorem}[section]
\newtheorem{cor}[thm]{Corollary}
\newtheorem{prop}[thm]{Proposition}
\newtheorem{lem}[thm]{Lemma}
\newtheorem{defn}[thm]{Definition}
\theoremstyle{definition} 
\newtheorem{exmp}[thm]{Example}
\newtheorem{rem}[thm]{Remark}

\newcommand{\Z}{\mathbb{Z}}
\newcommand{\Oo}{\mathcal{O}}
\newcommand{\ts}{\mathrm{ts}}
\newcommand{\tee}{\mathfrak{t}}
\newcommand{\orth}{\boxslash}

\newcommand{\pos}{\textnormal{Pos}}
\newcommand{\op}{\textnormal{op}}
\newcommand{\gr}{\Gamma}
\newcommand{\lex}{\times_{\textnormal{lex}}}

\DeclareMathOperator{\fib}{fib}
\DeclareMathOperator{\cofib}{cofib}

\begin{document}

\title{The (he)art of gluing}

\author[1]{Domenico Fiorenza\thanks{fiorenza@mat.uniroma1.it}}
\author[2]{Giovanni Luca Marchetti\thanks{gmarchetti1@sheffield.ac.uk}}

\affil[1]{\small{
	Dipartimento di Matematica, 
	Sapienza Universit\`a di Roma}}
\affil[2]{\small{
	School of Mathematics and Statistics, 
	University of Sheffield}
	}

\maketitle

\abstract{We introduce a notion of gluability for poset-indexed Bridgeland slicings on triangulated categories and show how a gluing abelian slicing on the heart of a bounded $t$-structure naturally induces a family of perverse $t$-structures. Our setup generalises the one of Collins and Polishchuk. As a corollary we recover several constructions from the theory of $t$-structures on triangulated categories. In particular we rediscover Levine's theorem: the Beilinson-Soul\'e vanishing conjecture implies the existence of Tate motives. }

\tableofcontents

\section{Prologue}
A common feature shared by several constructions involving $t$-structures on triangulated categories is the following. One starts with a (possibly infinite) semiorthogonal decomposition $\{\mathscr{D}_n\}_{n\in I\subseteq \Z}$ whose triangulated subcategories $\mathscr{D}_n$ are endowed with distinguished $t$-structures and, thanks to the vanishing of certain Ext groups, ends up with a new $t$-structure on $\mathscr{D}$ together with an Harder-Narashiman-type filtration on its heart. In terms of Bridgeland slicings (in the slightly generalized version presented in \cite{gkr,fosco}), this corresponds to trading a $\hat{\Z}\times_{\mathrm{lex}}\Z$-slicing for a ${\Z}\times_{\mathrm{lex}}\hat{\Z}$-slicing, where $\hat{\Z}$ denotes the ordered set of integers (with its usual order) endowed with the trivial $\Z$-action, and the subscript `lex' means that the product is given the lexicographic order. The aim of this note is to characterize the condition allowing this `exchange of factors', which we call `gluability' {as it is a direct generalization of the gluing of stability conditions introduced by Collins and Polishchuk in \cite{collins}, which is in turn} reminiscent of the `recollement situations' considered in \cite{bbd}. Once the gluability condition has been explicited, it is pretty immediate to identify several examples from the literature where this condition is more or less implicitly occurring. In particular, we are able to rediscover Levine's theorem: the Beilinson-Soul\'e vanishing conjecture implies the existence of Tate motives \cite{levine}. 
Other examples include Beilinson's construction of a distinguished $t$-stucture from a filtered structure on a triangulated category \cite{beil}, the $t$-structure obtained by Marcr\`i from exceptional collections with Ext vanishings in \cite{macri}, and the construction of a distinguished $t$-structure on the Fukaya category of the trivial $\mathbb{C}$-bundle over a symplectic manifold $M$ exhibited by Hensel in \cite{lagra}. Moreover, we show how the distinguished $t$-structure built out of a gluable slicing always come with a whole family of `perverse'' variants. For instance, the `perverse  motives' considered in \cite{permot} and the `perverse coherent sheves' considered in \cite{bezr} arise this way. A closely related example is the construction of the exotic $t$-structure on the derived category of $A$-modules, for $A$ a Koszul algebra, obtained by Koszul duality \cite{kosz}.

\vskip .5 cm
In this note we assume he reader is familiar with the language of Bridgeland slicings \cite{bridgeland}, in its generalization for an arbitrary poset $J$ endowed with a $\Z$-action considered in \cite{gkr} and surveyed in \cite{fosco}. We refer to \cite{fosco} for the notation and the definition used here. In particular we use the language of stable $\infty$-categories; the reader who prefers not to use higher categories will find no difficulty in verbatim translating each of the statements and proofs presented here in  the more classical language of triangulated categories.

\section{The gluing procedure} 

Let $\mathscr{D}$ a stable $\infty$-category or, if one prefers a more classical setting, a triangulated category.
If one considers Bridgeland slicings indexed by arbitrary partially orderes sets (endowed with a compatible $\Z$-action) as in \cite{gkr}, then one can think of associating with any $\Z$-poset $J$ the set of $J$-slicings on $\mathscr{D}$. This defines a `slice functor' on the category of $\Z$-posets,
see \cite[Remark 3.12]{fosco}. The slice functor actually hides a greater potential. Indeed, even if we don't have a morphism of $\mathbb{Z}$-posets (but just a $\mathbb{Z}$-equivariant map), then under suitably conditions the slice functor can still be applied. \\

\subsection{$f$-compatible slicings}
We begin by recalling the construction of the slice functor; see \cite{fosco} for details.
Let $J$ be a $\Z$-toset, i.e., a totally ordered set together with a monotone action of $\Z$, that we will denote by $(n,x)\mapsto x+n$. A $J$-slicing on a stable $\infty$-category $\mathscr{D}$ is a morphism of $\Z$-posests $\tee\colon \Oo(J)\to \ts(\mathscr{D})$, where $\Oo(J)$ is the $\Z$-poset of the slicings of $J$ (i.e., the decompositions of $J$ into a disjoint union of a lower set $L$ and of an upper set $U$), and $\ts(\mathscr{D})$ is the $\Z$-poset of $t$-structures on $\mathscr{D}$.

Clearly, if $f\colon J\to J'$ is a morphism of $\Z$-tosets, then \[
f^{-1}\colon \{\text{subsets of $J'$}\}\to \{\text{subsets of $J$}\}
\]
 induces a $\Z$-equivariant morphism of $\Z$-posets $f^{-1}\colon \Oo(J')\to \Oo(J)$, and so composition with $f^{-1}$ gives a morphism
\begin{align*}
f_*\colon J\text{-slicings on $\mathscr{D}$}&\to J'\text{-slicings on $\mathscr{D}$}\\
\tee&\mapsto \tee\circ f^{-1}.
\end{align*}
The slices of $f_*\tee$ are given by $\mathscr{D}_{f_*\tee;j}=\mathscr{D}_{\tee;f^{-1}(\{j\})}$, for any $j\in J'$. Notice that, as $f$ is monotone, the subset $f^{-1}(\{j\})$ is an interval in $J$. It is immediate to see that $f_*$ restricts to a map
\begin{align*}
f_*\colon \text{Bridgeland $J$-slicings on $\mathscr{D}$}&\to \text{Bridgeland $J'$-slicings on $\mathscr{D}$}.
\end{align*}
Namely, if $\mathcal{H}^j_{f_*\tee}(X)=0$ for every $j\in J'$ then $\mathcal{H}^{f^{-1}(\{j\})}_{\tee}(X)=0$ for every $j$ in $J'$. As we are assuming the $J$-slicing $\tee$ is a Bridgeland slicing, this implies that $\mathcal{H}^{\phi}_{\tee}(X)=0$ for every $\phi$ in $f^{-1}(\{j\})$, for every $j$. Therefore $\mathcal{H}^{\phi}_{\tee}(X)=0$ for every $\phi$ in $J$ and so, again by definition of Bridgeland slicing, $X=0$. Also, if $\mathcal{H}^j_{f_*\tee}(X)\neq 0$ then $\mathcal{H}^{f^{-1}(\{j\})}_{\tee}(X)\neq 0$ and so (again by the Bridgeland slicing condition) there exists at least an element $\phi$ in $f^{-1}(\{j\})$ such that $\mathcal{H}^{\phi}_{\tee}(X)\neq 0$. As the $J$-slicing $\tee$ is Bridgeland, the total of these $\phi$'s must be finite, so only for finitely many $j$ we can have such a $\phi$. In other words, the number of indices $j$ in $J'$ such that $\mathcal{H}^j_{f_*\tee}(X)\neq 0$ is finite.

Notice that, if $\tee$ is a Bridgeland slicing of $\mathscr{D}$, and $f\colon J\to J'$ is a morphism of $\Z$-tosets, then for any slicing $(L,U)$ of $J'$, the lower and the upper categories
$\mathscr{D}_{f_*\tee;L}$ and $\mathscr{D}_{f_*\tee;U}$ can be equivalently defined as
\begin{align*}
\mathscr{D}_{f_*\tee;L}&=\langle \mathscr{D}_{\tee;\phi}\rangle_{f(\phi)\in L}\\
\mathscr{D}_{f_*\tee;U}&=\langle \mathscr{D}_{\tee;\phi}\rangle_{f(\phi)\in U}
\end{align*}
where $\langle \mathscr{S}\rangle$ denotes the extension-closed subcategory of $\mathscr{D}$ generated by the subcategory $\mathscr{S}$. In particular the slices of are given by 
$\mathscr{D}_{f_*\tee;j}=\langle \mathscr{D}_{\tee;\phi}\rangle_{f(\phi)=j}$.\\

\begin{rem}
The right hand sides of the above two expressions can clearly be defined for every morphism $f$ from $J$ to $J'$ (i.e., not necessarily monotone nor $\Z$-equivariant), and as soon as $f$ is $\Z$-equivariant, the assignment
\[
(L,U)\mapsto (\langle \mathscr{D}_{\tee;\phi}\rangle_{f(\phi)\in L},\langle \mathscr{D}_{\tee;\phi}\rangle_{f(\phi)\in U})
\]
is an equivariant morphism from $\mathcal{O}(J)$ to pairs of subcategories of $\mathscr{D}$. Clearly, when $f$ is not monotone there is no reason to expect that the pair $(\langle \mathscr{D}_{\tee;\phi}\rangle_{f(\phi)\in L},\langle \mathscr{D}_{\tee;\phi}\rangle_{f(\phi)\in U})$ forms a $t$-structure on $\mathscr{D}$.
\end{rem}
Yet, it interesting to notice that the condition that $f$ be monotone is only only sufficient in order to have this, and can indeed be relaxed.\\

\begin{defn}\label{compatible}
Let $J$ and $J'$ be $\Z$-tosets, and let $f\colon J\to J'$ a map of $\Z$-sets (i.e., a $\Z$-equivariant map, not necessarily nondecreasing). A Bridgeland $J$-slicing $\tee$ of $\mathscr{D}$ is said to be \emph{$f$-compatible}
if the condition `$f(\phi)>f(\psi)$ with $\phi\leq \psi$' implies $\mathscr{D}_{\tee; \phi}\orth \mathscr{D}_{\tee;\psi}$ and $\mathscr{D}_{\tee; \phi}\orth \mathscr{D}_{\tee;\psi}[1]$.
\footnote{Given two subcategories $\mathscr{S}_1$ and $\mathscr{S}_2$ of $\mathscr{D}$ we write $\mathscr{S}_1\orth \mathscr{S}_2$ to mean that $\mathscr{D}(X_1,X_2)$ is contractible for any $X_1\in \mathscr{S}_1$ and any $X_2\in\mathscr{S}_2$. It is easy to see that $\mathscr{S}_1\orth \mathscr{S}_2$ implies $\langle\mathscr{S}_1\rangle\orth \langle\mathscr{S}_2\rangle$, see \cite[Lemma 4.21]{fosco}}
\end{defn}

\begin{rem}\label{everything-compatible}
Clearly, if $f$ is monotone, then every $J$-slicing $\tee$ is $f$-compatible as the condition `$f(\phi)>f(\psi)$ with $\phi\leq \psi$' is empty.
\end{rem}

\begin{rem}\label{avanti-e-indietro}
Let $J$ and $J'$ be $\Z$-tosets, let $f\colon J\to J'$ a map of $\Z$-sets, and let $g\colon J'\to J''$ be an isomorphism of $\Z$-posets. Then a Bridgeland $J$-slicing $\tee$ of $\mathscr{D}$ is $f$-compatible if and only if it is $(g\circ f)$-compatible. Similarly, if $h\colon J''\to J$ is an isomorphism of $\Z$-posets, then $\tee$ is $f$-compatible if and only if $(h^{-1})_*\tee$ is $(f\circ h)$-compatible.
\end{rem}

\begin{lem}\label{is-t-structure}
Let $J$ and $J'$ be $\Z$-tosets, and let $f\colon J\to J'$ be a $\Z$-equivariant morphism of $\Z$-sets (i.e., not necessarily a monotone map) and let $\tee$ be a Bridgeland slicing of $\mathscr{D}$ which is $f$-compatible. Then, for any slicing $(L,U)$ of $J'$, the pair of subcategories
$(\langle \mathscr{D}_{\tee;\phi}\rangle_{f(\phi)\in L},\langle \mathscr{D}_{\tee;\phi}\rangle_{f(\phi)\in U})$ is a $t$-structure on $\mathscr{D}$.
\end{lem}
\begin{proof}
As $f$ is $\Z$-equivariant and $U+1\subseteq U$, we have
\begin{align*}
\langle \mathscr{D}_{\tee;\phi}\rangle_{f(\phi)\in U}[1]&=\langle \mathscr{D}_{\tee;\phi}[1]\rangle_{f(\phi)\in U}\\
&=\langle \mathscr{D}_{\tee;\phi+1}\rangle_{f(\phi)\in U}\\
&=\langle \mathscr{D}_{\tee;\phi+1}\rangle_{f(\phi+1)\in U+1}\\
&=\langle \mathscr{D}_{\tee;\phi}\rangle_{f(\phi)\in U+1}\\
&\subseteq \langle \mathscr{D}_{\tee;\phi}\rangle_{f(\phi)\in U}
\end{align*}
and similarly for the lower subcategory $\langle \mathscr{D}_{\tee;\phi}\rangle_{f(\phi)\in L}$. To show that 
\[
\langle \mathscr{D}_{\tee;\phi}\rangle_{f(\phi)\in U}\orth \langle \mathscr{D}_{\tee;\phi}\rangle_{f(\phi)\in L}
\]
it suffices to show that, if $f(\psi)\in L$ and $f(\phi)\in U$ then $\mathscr{D}_{\tee;\phi}\orth \mathscr{D}_{\tee;\psi}$. As $L\cap U=\emptyset$ we cannot have $\psi=\phi$, so either $\psi<\phi$ or vice versa. In the first case, $\mathscr{D}_{\tee;\phi}\orth \mathscr{D}_{\tee;\psi}$ by definition of Bridgeland $J$-slicing. In the second case, we have $\phi<\psi$ and $f(\phi)>f(\psi)$ as $f(\phi)\in U$ and  $f(\psi)\in L$. Therefore, since $\tee$ is $f$-compatible, $\mathscr{D}_{\tee;\phi}\orth\mathscr{D}_{\tee;\psi}$. Finally, we have to show that every object $X$ in $\mathscr{D}$ fits into a fiber sequence
\[
\xymatrix{
X_U\ar[r]\ar[d] & X\ar[d]\\
0\ar[r] & X_L
}
\]
with $X_L\in \langle \mathscr{D}_{\tee;\phi}\rangle_{f(\phi)\in L}$ and $X_U\in \langle \mathscr{D}_{\tee;\phi}\rangle_{f(\phi)\in U}$. As $\tee$ is a Bridgeland slicing, we have a factorization of the initial morphism $\mathbf{0} \to X$ of the form
\[
\mathbf{0}=X_0 \xrightarrow{\alpha_1} X_1\cdots \xrightarrow{\alpha_{{\bar\imath}}}X_{{\bar\imath}}\xrightarrow{\alpha_{{\bar\imath}+1}}X_{{\bar\imath}+1}\xrightarrow{}\cdots \xrightarrow{\alpha_n} X_n=X
\]
with $\mathbf{0} \neq \cofib(\alpha_i)=\mathcal{H}_{\tee}^{\phi_i}(X) \in \mathscr{D}_{\phi_i}$ for all $i = 1, \cdots, n$, with $\phi_i>\phi_{i+1}$. Let us now consider the sequence of symbols $L$ and $U$ obtained putting in the $i$-th place $L$ if $f(\phi_i)\in L$ and $U$ if $f(\phi_i)\in U$. If this sequence is of the form $(U,U,\dots,U,L,L,\dots,L)$, then there exists an index $\bar{\imath}$ such that
 $f(\phi_{i})\in U$ for $i\leq \bar{\imath}$ and $f(\phi_{i})\in L$ for $i>\bar{\imath}$ (with ${\bar\imath}=-1$ or $n$ when all of the $f(\phi_i)$ are in $L$ or in $U$, respectively). Then we can
consider the pullout diagram
\[
\xymatrix{
X_{{\bar\imath}}\ar[r]\ar[d]_{f_{L}}&0\ar[d]\\
X\ar[r]&\mathrm{cofib}(f_{L})
}\]
together with the factorizations
\[
\mathbf{0}=X_0 \xrightarrow{\alpha_1} X_1\cdots \xrightarrow{\alpha_{{\bar\imath}}}X_{{\bar\imath}} 
\]
and
\[
X_{{\bar\imath}}\xrightarrow{\alpha_{{\bar\imath}+1}}X_{{\bar\imath}+1}\xrightarrow{}\cdots \xrightarrow{\alpha_n} X_n=X.
\]
The first factorization shows that $X_{{\bar\imath}}\in \langle\cup_{i=0}^{{\bar\imath}}\mathscr{D}_{\phi_i}\rangle\subseteq \langle \mathscr{D}_{\tee;\phi}\rangle_{f(\phi)\in U}$ while the second factorization shows that $\mathrm{cofib}(f_{L})\in  \langle\cup_{i={\bar\imath}+1}^n\mathscr{D}_{\phi_i}\rangle\subseteq \langle \mathscr{D}_{\tee;\phi}\rangle_{f(\phi)\in L}$. So we are done in this case. Therefore, we are reduced to showing that we can always avoid a $(\dots,L,U,\dots)$ situation in our sequence of $L$'s and $U$'s. Assume we have such a situation. Then we have an index $i_0$ with $f(\phi_{i_0})\in L$ and $f(\phi_{{i_0}+1})\in U$. This in particular implies $f(\phi_{{i_0}+1})>f(\phi_{i_0})$ with $\phi_{{i_0}+1}<\phi_{i_0}$. As $\tee$ is $f$-compatible, this gives $\mathscr{D}_{\tee; \phi_{i_0+1}}\orth (\mathscr{D}_{\tee;\phi_{i_0}}[1])$. In particular, $\mathscr{D}(\mathcal{H}^{\phi_{i_0+1}}_{\tee}(X),\mathcal{H}^{\phi_{i_0}}_{\tee}(X)[1])$ is contractible. Now consider the pasting of pullout diagrams
\[
\xymatrix{X_{i_0-1}\ar[r]^{\alpha_{i_0}}\ar[d] & X_{i_0}\ar[r]^{\alpha_{i_0+1}}\ar[d] & X_{i_0+1}\ar[d]^{\gamma}\\
0\ar[r]&\mathcal{H}^{\phi_{i_0}}_{\tee}(X)\ar[d]\ar[r]&Y\ar[d]\ar[r] &0\ar[d]\\
&0\ar[r]&\mathcal{H}^{\phi_{i_0+1}}_{\tee}(X)\ar[r]&\mathcal{H}^{\phi_{i_0}}_{\tee}(X)[1].
}
\]
As the arrow $\mathcal{H}^{\phi_{i_0+1}}_{\tee}(X)\to \mathcal{H}^{\phi_{i_0}}_{\tee}(X)[1]$ factors through $0$, we have $Y\cong\mathcal{H}^{\phi_{i_0+1}}_{\tee}(X)\oplus \mathcal{H}^{\phi_{i_0}}_{\tee}(X)$ and the above diagram becomes
\[
\xymatrix{X_{i_0-1}\ar[r]^{\alpha_{i_0}}\ar[d] & X_{i_0}\ar[r]^{\alpha_{i_0+1}}\ar[d] & X_{i_0+1}\ar[d]^{\gamma}\\
0\ar[r]&\mathcal{H}^{\phi_{i_0}}_{\tee}(X)\ar[d]\ar[r]^-{\iota_2}&\mathcal{H}^{\phi_{i_0+1}}_{\tee}(X)\oplus \mathcal{H}^{\phi_{i_0}}_{\tee}(X)\ar[d]^{\pi_1}\ar[r] &0\ar[d]\\
&0\ar[r]&\mathcal{H}^{\phi_{i_0+1}}_{\tee}(X)\ar[r]&\mathcal{H}^{\phi_{i_0}}_{\tee}(X)[1],
}
\]
where $\iota_2$ and $\pi_1$ are the canonical inclusion and projection. Let $\beta\colon X_{i_0+1}\to \mathcal{H}^{\phi_{i_0}}_{\tee}(X)$ be the composition
\[
\beta\colon X_{i_0+1}\xrightarrow{\gamma} \mathcal{H}^{\phi_{i_0+1}}_{\tee}(X)\oplus \mathcal{H}^{\phi_{i_0}}_{\tee}(X)\xrightarrow{\pi_2} 
\mathcal{H}^{\phi_{i_0}}_{\tee}(X).
\]
Then we have a homotopy commutative diagram
\[
\xymatrix{
X_{i_0-1}\ar[r]^{\alpha_{i_0}}\ar[d]&X_{i_0}\ar[r]^{\alpha_{i_0+1}}\ar[d] & X_{i_0+1}\ar[d]^{\beta}\\
0\ar[r]&\mathcal{H}^{\phi_{i_0}}_{\tee}(X)\ar[r]^-{\mathrm{id}}&\mathcal{H}^{\phi_{i_0}}_{\tee}(X)
}
\]
(where only the left square is a pullout), and so the composition $\alpha_{i_0+1}\circ \alpha_{i_0}$ factors through the homotopy fiber of $\beta$. In other words, we have a homotopy commutative diagram
\[
\xymatrix{
X_{i_0-1}\ar[r]^{\alpha_{i_0}}\ar[d]_{\tilde{\alpha}_{i_0}}&X_{i_0}\ar[d]^{\alpha_{i_0+1}}\\
\fib(\beta)\ar[r]^{\tilde{\alpha}_{i_0+1}}&X_{i_0+1}\
}
\]
Writing $\tilde{X}_{i_0}=\fib(\beta)$, we get the pasting of pullout diagrams
\[
\xymatrix{X_{i_0-1}\ar[r]^{\tilde{\alpha}_{i_0}}\ar[d] & \tilde{X}_{i_0}\ar[r]^{\tilde{\alpha}_{i_0+1}}\ar[d] & X_{i_0+1}\ar[d]^{\gamma}\\
0\ar[r]&\mathcal{H}^{\phi_{i_0+1}}_{\tee}(X)\ar[d]\ar[r]^-{\iota_1}&\mathcal{H}^{\phi_{i_0+1}}_{\tee}(X)\oplus \mathcal{H}^{\phi_{i_0}}_{\tee}(X)\ar[d]_{\pi_2}\\
&0\ar[r]&\mathcal{H}^{\phi_{i_0}}_{\tee}(X).
}
\]
That is, by considering the factorization $X_{i_0-1}\xrightarrow{\tilde{\alpha}_{i_0}} \tilde{X}_{i_0}\xrightarrow{\tilde{\alpha}_{i_0+1}} X_{i_0+1}$ we have switched the cofibers with respect to the original factorization $X_{i_0-1}\xrightarrow{{\alpha}_{i_0}} {X}_{i_0}\xrightarrow{{\alpha}_{i_0+1}} X_{i_0+1}$. Therefore, writing $\tilde{\phi}_{i_0}=\phi_{i_0+1}$ and $\tilde{\phi}_{i_0+1}=\phi_{i_0}$, 
we now have  $f(\tilde{\phi}_{i_0})\in U$ and $f(\tilde{\phi}_{i_0+1})\in L$. That is, we have removed the $(\dots,L,U,\dots)$ situation from the position $i_0$, replacing it with a $(\dots,U,L,\dots)$ situation, while keeping all the labels $L,U$ before this positions unchanged. Repeating the procedure the needed number of times, we eventually get rid of all the $(\dots,L,U,\dots)$ situations.\footnote{This is somehow reminiscent of the `bubble sort' algorithm.}
\end{proof}

\begin{rem}\label{a-closer-look}
It follows from the proof of Lemma \ref{is-t-structure} that the objects $X_L$ and $X_U$ in the fiber sequence $X_U\to X\to X_L$ associated with the $t$-struture $(\langle \mathscr{D}_{\tee;\phi}\rangle_{f(\phi)\in L},\langle \mathscr{D}_{\tee;\phi}\rangle_{f(\phi)\in U})$ on $\mathscr{D}$ satisfy
\begin{align*}
X_L&\in \langle \mathscr{D}_{\tee;\phi};\quad f(\phi)\in L \text{ and }\mathcal{H}^\phi_\tee(X)\neq 0\rangle\\
X_U&\in \langle \mathscr{D}_{\tee;\phi};\quad f(\phi)\in U \text{ and }\mathcal{H}^\phi_\tee(X)\neq 0\rangle
\end{align*}
\end{rem}

\begin{lem}\label{slices}
For any $j\in J'$ we have
\[
\langle \mathscr{D}_{\tee;\phi}\rangle_{f(\phi)=j}=\langle \mathscr{D}_{\tee;\phi}\rangle_{f(\phi)\leq j}\cap \langle \mathscr{D}_{\tee;\phi}\rangle_{f(\phi)\geq j}.
\]
\end{lem}
\begin{proof}
Clearly, $\langle \mathscr{D}_{\tee;\phi}\rangle_{f(\phi)=j}\subseteq \langle \mathscr{D}_{\tee;\phi}\rangle_{f(\phi)\leq j}\cap \langle \mathscr{D}_{\tee;\phi}\rangle_{f(\phi)\geq j}$, therefore we only need to prove the converse inclusion. Let $\in \langle \mathscr{D}_{\tee;\phi}\rangle_{f(\phi)\leq j}\cap \langle \mathscr{D}_{\tee;\phi}\rangle_{f(\phi)\geq j}$. Then in particular 
$X\in \langle \mathscr{D}_{\tee;\phi}\rangle_{f(\phi)\geq j}$
and so there exists a factorization of the initial morphism $\mathbf{0} \to X$ of the form
\[
\mathbf{0}=X_0 \xrightarrow{\alpha_1} X_1\xrightarrow{\alpha_2}\cdots \xrightarrow{\alpha_{n-1}}X_{n-1}\xrightarrow{\alpha_n} X_n=X
\]
with $\mathbf{0} \neq \cofib(\alpha_i) \in \mathscr{D}_{\phi_i}$ with $f(\phi_i)\geq j$ for all $i = 1, \cdots, n$. As $((-\infty,j],(j,+\infty)$ is a slicing of $J'$, reasoning as in the proof of Lemma \ref{is-t-structure} we can arrange this factorization is such a way that $f(\phi_i)>j$ for $i\leq \bar{\imath}$ and $f(\phi_i)= j$ for $i>\bar{\imath}$. Therefore, again by reasoning as in the proof of Lemma \ref{is-t-structure} we get a fiber sequence of the form
\[
\xymatrix{
X_{>j}\ar[r]\ar[d] & X\ar[d]\\
0\ar[r]&X_{j}
}
\]
with $X_{>j}$ in $\langle \mathscr{D}_{\tee;\phi}\rangle_{f(\phi)>j}$ and $X_j$ in $\langle \mathscr{D}_{\tee;\phi}\rangle_{f(\phi)=j}$. This is in particular a fiber sequence of the form $X_{>j}\to X\to X_{\leq j}$, with $X_{\leq j}\in \langle \mathscr{D}_{\tee;\phi}\rangle_{f(\phi)\leq j}$. 
As $(\langle \mathscr{D}_{\tee;\phi}\rangle_{f(\phi)\leq j},\langle \mathscr{D}_{\tee;\phi}\rangle_{f(\phi)> j})$ is a $t$-structure on $\mathscr{D}$ by Lemma \ref{is-t-structure}, there is (up to equivalence) only one such a fiber sequence. And since $X\in \langle \mathscr{D}_{\tee;\phi}\rangle_{f(\phi)\leq j}$, this is the sequence $0\to X\xrightarrow{\mathrm{id}} X$. Therefore $X=X_j$.
\end{proof}

\begin{prop}
Let $J,J'$ be $\Z$-tosets, and let  $f\colon J\to J'$ be a $\Z$-equivariant morphism of $\Z$-sets (i.e., not necessarily a monotone map) and let $\tee$ be a Bridgeland slicing of $\mathscr{D}$ which is $f$-compatible.  The map
\begin{align*}
f_!\tee\colon \Oo(J')&\to\ts(\mathscr{D})\\
(L,U)&\mapsto (\langle \mathscr{D}_{\tee;\phi}\rangle_{f(\phi)\in L},\langle \mathscr{D}_{\tee;\phi}\rangle_{f(\phi)\in U})
\end{align*}
defined by Lemma \ref{is-t-structure} is a Bridgeland $J'$-slicing of $\mathscr{D}$, with slices given by
\[
\mathscr{D}_{f_!\tee;j}=\langle \mathscr{D}_{\tee;\phi}\rangle_{f(\phi)=j}.
\]
\end{prop}
\begin{proof}
The map $f_!\tee$ is manifestly monotone and $\Z$-equivariant (see the first part of the proof of Lemma \ref{is-t-structure}), so it is a $J'$ slicing of $\mathscr{D}$, and its slices are given by $\mathscr{D}_{f_!\tee;j}=
%\mathscr{D}_{f_!\tee;\geq j}\cap \mathscr{D}_{f_!\tee;\leq j},
%\]
%and so by the subcategories $
\langle \mathscr{D}_{\tee;\phi}\rangle_{f(\phi)=j}$ by Lemma \ref{slices}. We are therefore left with showing that it is finite and discrete. Given an object $X$ in $\mathscr{D}$, let  $\{\phi_1,\dots,\phi_n\}$ be the indices in $J$ such that $\mathcal{H}^{\phi_i}_\tee(X)\neq 0$ and let $\{j_1,\dots,j_k\}$ the image of the set $\{\phi_1,\dots,\phi_n\}$ via $f$. Up to renaming, we can assume $j_1>j_{2}>\cdots>j_k$. Consider now the factorization
\[
\mathbf{0}=X_0 \xrightarrow{\alpha_1} X_1\xrightarrow{\alpha_2}\cdots \xrightarrow{\alpha_{k-1}}X_{k-1}\xrightarrow{\alpha_k} X_k=X
\]
of the initial morphism of $X$ associated to the decreasing sequence $j_1>j_2>\cdots>j_k$ by the $J'$-slicing $f_!\tee$. The cofibers of the morphisms $\alpha_{i}$ are the cohomologies
$\mathcal{H}^{(j_{i+1},j_{i}]}_{f_!\tee}(X)$ and, by Remark \ref{a-closer-look}, we have
\begin{align*}
\mathcal{H}^{(j_{i+1},j_{i}]}_{f_!\tee}(X)\in& \langle \mathscr{D}_{\tee;\phi};\quad f(\phi)\in (j_{i+1},j_{i}] \text{ and }\mathcal{H}^\phi_\tee(X)\neq 0\rangle\\
&=\langle \mathscr{D}_{\tee;\phi};\quad \phi\in\{\phi_1,\dots,\phi_n\}\text{ and } f(\phi)\in(j_{i+1},j_{i}]\rangle\\
&=\langle \mathscr{D}_{\tee;\phi};\quad \phi\in\{\phi_1,\dots,\phi_n\}\text{ and } f(\phi)=j_{i}\rangle\\
&\subseteq \langle \mathscr{D}_{\tee;\phi};\quad  f(\phi)=j_{i}\rangle=\mathscr{D}_{f_!\tee;j_{i}}.
\end{align*}
Therefore
\[
\mathcal{H}^{(j_i,j_{i-1}]}_{f_!\tee}(X)=\mathcal{H}^{j_i}_{f_!\tee}\mathcal{H}^{(j_i,j_{i-1}]}_{f_!\tee}(X)=\mathcal{H}^{j_i}_{f_!\tee}(X).
\]
This tells us that, if all the cohomologies $\mathcal{H}^{j}_{f_!\tee}(X)$ vanish, then also all the $\mathcal{H}^{(j_i,j_{i-1}]}_{f_!\tee}(X)$ vanish, and so $X=0$. Finally, if $\tilde{j}\notin\{j_1,\dots,j_k\}$, then there exists an index $i$ such that $j_{i+1}<\tilde{j}<j_{i}$ and so $(j_{i+1},\tilde{j}]\cap \{j_1,\dots,j_k\}=\emptyset$. The above argument then shows that
\[
\mathcal{H}^{(j_{i+1},\tilde{j}]}_{f_!\tee}(X)\in\langle \mathscr{D}_{\tee;\phi};\quad \phi\in\{\phi_1,\dots,\phi_n\}\text{ and } f(\phi)\in(j_{i+1},\tilde{j}]\rangle=\{\mathbf{0}\}.
\]
It follows that 
\[
\mathcal{H}^{\tilde{j}}_{f_!\tee}(X)=\mathcal{H}^{\tilde{j}}_{f_!\tee}\mathcal{H}^{(j_{i+1},\tilde{j}]}_{f_!\tee}(X)=\mathcal{H}^{\tilde{j}}_{f_!\tee}(0)=0,
\]
for every $\tilde{j}\notin \{j_1,\dots,j_k\}$. So in particular, for any $X\in \mathscr{D}$, the cohomologies $\mathcal{H}^{j}_{f_!\tee}(X)$ are possibly nonzero only for finitely many indices $j$. 
\end{proof}

\begin{rem}
If $f\colon J\to J'$  is a morphism of $\Z$-tosets, then every Bridgeland slicing of $\tee$ of $\mathscr{D}$ is $f$-compatible and one has $f_!\tee=f_*\tee$.
\end{rem}

\begin{prop}\label{functoriality}
Let $J,J',J''$ be $\Z$-tosets, let $f\colon J\to J'$ and $g\colon J'\to J''$ be $\Z$-equivariant morphisms of $\Z$-sets (i.e., not necessarily monotone maps), and let $\tee$ be a Bridgeland slicing of $\mathscr{D}$. If $\tee$ is $f$-compatible and $f_!\tee$ is $g$-compatible, then $\tee$ is $(g\circ f)$-compatible, and
\[
(g\circ f)_!\tee= g_!f_!\tee.
\]
\end{prop}
\begin{proof}
Let $\phi,\psi$ in $J$ be such that $\phi\leq \psi$ and $g(f(\phi))>g(f(\psi))$, and let $\xi=f(\phi)$ and $\eta=f(\psi)$. As $J'$ is a totally ordered set, either $\xi\leq \eta$ or $\xi>\eta$. In the first case we have $g(\xi)>g(\eta)$ with $\xi\leq \eta$. As $f_!\tee$ is $g$-compatible, this implies $\mathscr{D}_{f_!\tee; \xi}\orth \mathscr{D}_{f_!\tee;\eta}$ and $\mathscr{D}_{f_!\tee; \xi}\orth \mathscr{D}_{f_!\tee;\eta}[1]$. By definition of $f_!\tee$ we have $\mathscr{D}_{\tee; \phi}\subseteq \mathscr{D}_{f_!\tee; \xi}$ and $\mathscr{D}_{\tee; \psi}\subseteq \mathscr{D}_{f_!\tee; \eta}$, therefore we have $\mathscr{D}_{\tee; \phi}\orth \mathscr{D}_{\tee;\psi}$ and $\mathscr{D}_{\tee; \phi}\orth \mathscr{D}_{\tee;\psi}[1]$. In the second case we have $f(\phi)>f(\psi)$ with $\phi\leq \psi$. As $\tee$ is $f$-compatible, this again
implies $\mathscr{D}_{\tee; \phi}\orth \mathscr{D}_{\tee;\psi}$ and $\mathscr{D}_{\tee; \phi}\orth \mathscr{D}_{\tee;\psi}[1]$. Finally, for any $j\in J''$ one has
\[
\mathscr{D}_{g_!f_!\tee;j}=\langle \mathscr{D}_{f_!\tee;\xi}\rangle_{g(\xi)=j}=\langle\langle \mathscr{D}_{\tee;\phi}\rangle_{f(\phi)=\xi}\rangle_{g(\xi)=j}=\langle \mathscr{D}_{\tee;\phi}\rangle_{g(f(\phi))=j}=\mathscr{D}_{(g\circ f)_!\tee;j}.
\]
\end{proof}

\subsection{The exchange map}

Let now $J_1$ and $J_2$ be two $\mathbb{Z}$-tosets, and let $J_1 \lex J_2$ and $J_2 \lex J_1$ the two $\Z$-tosets obtained by considering the lexicogaphic order on the products $J_1\times J_2$ and $J_2\times J_1$, respectively, and the diagonal $\Z$ action. The following lemma is immediate.
\begin{lem}\label{exchange}
The exchange map
\begin{align*}
e\colon J_1\times J_2&\to J_2\times J_1\\
(j_1,j_2)&\mapsto (j_2,j_1)
\end{align*}
is $\Z$-equivariant.
\end{lem}

\begin{defn}\label{gluable}
Let $J_1$ and $J_2$ be $\Z$-tosets. A Bridgeland $J_1 \lex J_2$-slicing $\tee$ of $\mathscr{D}$ is said to be \emph{gluable} if it is $e$-compatible, where $e \colon J_1 \times J_2 \to J_2 \times J_1$ is the exchange map. 
\end{defn}

\begin{exmp}\label{example:bbd} Let $J_1=\{0,1\}$ with the trivial $\Z$-action and let $J_2=\Z$ with the standard translation $\Z$-action.  Then a $J_1\lex J_2$-slicing $\tee$ on $\mathscr{D}$ is the datum of a semiorthogonal decomposition $(\mathscr{D}_0,\mathscr{D}_1)$ of $\mathscr{D}$ toghether with $t$-structures $\tee_i$ on $\mathscr{D}_i$. Spelling out the definition, one sees that the $\{0,1\}\lex \Z$-slicing $\tee$ is gluable if and only if 
$\mathscr{D}_{0;\geq 0}\orth \mathscr{D}_{1;-1}$ and $\mathscr{D}_{0;\geq 0}\orth \mathscr{D}_{1;0}$, and so if and only if $\mathscr{D}_{0;\geq 0}\orth \mathscr{D}_{1;0}$. {As  $\mathscr{D}_{0;\geq 0}$ is generated by the subcategories $\mathscr{D}_{0}^\heartsuit[k]$ for $k\geq 0$, this is equivalent to $\mathscr{D}_{0}^\heartsuit\orth\mathscr{D}_{1}^\heartsuit[n]$ for all $n\leq 0$.} When this happens, the glued slicing $e_!\tee$ is a $\Z\lex\{0,1\}$ slicing on $\mathscr{D}$, i.e., is the datum of a bounded $t$-structure on $\mathscr{D}$ together with a torsion theory on its heart. More precisely, the heart of $e_!\tee$ is the full $\infty$-subcategory $\mathscr{D}^{\heartsuit_{e_!\tee}}$ of $\mathscr{D}$ on those objects $X$ that fall into fiber sequences of the form $X_1\to X\to X_0$ with $X_i\in \mathscr{D}_i^\heartsuit$, while the torsion theory on $\mathscr{D}^{\heartsuit_{e_!\tee}}$ is
  $(\mathscr{D}_0^{\heartsuit},\mathscr{D}_1^{\heartsuit})$, i.e., precisely the pair of hearts of the two subcategories in the semiorthogonal decomposition. \\
  This is exactly the setup of \cite{collins}, which on the other hand is a particular case of the classical gluing construction of $t$-structures by \cite{bbd}, explaining both our nomenclature and motivation.
\end{exmp}

\begin{lem}\label{incolla2}
Let $J_1$  and $J_2$ be $\Z$-tosets, and let $\tee$ be a Bridgeland $J_1 \lex J_2$-slicing of $\mathscr{D}$. If $\tee$ is gluable and $\Z$ acts trivially on $J_1$, then $e_! \tee$ is a gluable $J_2 \lex J_1$ slicing of $\mathscr{D}$. 
\end{lem}
\begin{proof}
Let $(j_2,j_1)\leq (j_2',j_1')$ in $J_2 \lex J_1$ be such that $(j_1,j_2)>(j_1',j_2')$ in $J_1 \lex J_2$. 
%We have four possibilities:
%\begin{itemize}
%\item $j_2<j_2'$ and $j_1>j_1'$;
%\item $j_2=j_2'$ and $j_1\leq j_1'$ and $j_1>j_1'$;
%\item $j_2<j_2'$ and $j_1=j_1'$ and $j_2>j_2'$;
%\item $j_2=j_2'$ and $j_1\leq j_1'$ and $j_1=j_1'$ and $j_2>j_2'$.
%\end{itemize}
As $\tee$ is gluable, we have
\[
\mathscr{D}_{e_!\tee;(j_2,j_1)}=\langle \mathscr{D}_{\tee;(i_1,i_2)}\rangle_{e(i_1,i_2)=(j_2,j_1)}=\langle \mathscr{D}_{\tee;(j_1,j_2)}\rangle=\mathscr{D}_{\tee;(j_1,j_2)},
\]
since slices are extension closed. As $(j_1,j_2)>(j_1',j_2')$, we have $\mathscr{D}_{\tee;(j_1,j_2)}\orth \mathscr{D}_{\tee;(j_1',j_2')}$ by definition of slicing, and so $\mathscr{D}_{e_!\tee;(j_2,j_1)}\orth\mathscr{D}_{e_!\tee;(j_2',j_1')}$. Finally, 
\[
\mathscr{D}_{e_!\tee;(j_2,j_1)}[1]=\mathscr{D}_{\tee;(j_1,j_2)}[1]=\mathscr{D}_{\tee;(j_1,j_2)+1}=\mathscr{D}_{\tee;(j_1,j_2+1)},
\]
since the $\Z$-action on $J_1$ is trivial. The condition $(j_2,j_1)\leq (j_2',j_1')$ with $(j_1,j_2)>(j_1',j_2')$ is actually equivalent to $j_2<j_2'$ and $j_1>j_1'$, so in particular we have $(j_1,j_2)>(j_1',j_2'+1)$.This implies $\mathscr{D}_{\tee;(j_1,j_2)}\orth \mathscr{D}_{\tee;(j_1',j_2'+1)}$ and so $\mathscr{D}_{e_!\tee;(j_2,j_1)}\orth\mathscr{D}_{e_!\tee;(j_2',j_1')}[1]$.
\end{proof}

\subsection{Upper graphs and monotone maps}

\begin{defn}
Let $U\colon J \to \Oo(J')$ be a map of sets. We denote by $\gr_{U}$ the subset of $J\times J'$ defined by 
 $$\gr_U=\{ (j,j') \in J \times J'; \quad j' \in U_{j} \}.$$
 For $f\colon J\to J'$ a map of sets, we denote by $(<f,\geq f)\colon J \to \Oo(J')$ the composition of $f$ with the map
 \begin{align*}
 J'&\to \Oo(J')\\
 j'&\mapsto ((-\infty,j'),[j',+\infty)).
 \end{align*}
 The upper graph of $f$ is the subset $\gr_{\geq f}$ of $J\times J'$.
\end{defn}
  
\begin{lem}\label{decreasing-gives-upper-set}
  Let $U \colon J \to \Oo(J')$ be a map of sets. Then $\gr_U$
  is an upper set of $J \times J'$ with the \emph{product order} if and only if $U$ is a map of posets $U\colon J\to \Oo(J')^\op$. In particular, $\gr_{\geq f}$ is an upper set if and only if $f\colon J\to J'$ is a nondecreasing map, i.e., a map of posets $J\to J'^\op$.
\end{lem}

\begin{proof}Assume $(L,U)$ is a map of posets from $J$ to $\Oo(J')^\op$.
Pick $(j,j') \in \gr_U$ and suppose $(j,j') \leq (k,k')$ in $J \times J'$. As $(L,U)$ is monotone and $k\geq j$, we have $U_j\subseteq U_k$; as $(j,j')\in \gr_U$, we have $j' \in U_j$. Therefore $j\in U_k$. Since $U_k$ is an upper set and $k'\geq j'$, we get $k'\in U_k$, i.e. $(k,k')\in \gr_U$. Vice versa, assume $\gr_U$ is an upper set, and let $j\leq k$ in $J$. For any $j'\in U_j$, the element $(k,j')$ in $J\times J'$ satisfies $(k,j')\geq (j,j')$ in the product order and so $(k,j')\in \gr_U$. This means $j'\in U_k$ and so $U_j\subseteq U_k$. To prove the second part of the statement, notice that the map $j\mapsto ((-\infty,j'),[j',+\infty))$ is a map of posets from $J'\to \Oo(J')$. Therefore, if $f\colon J\to J'^\op$ is a map of posets, then also $(<f,\geq f)\colon J \to \Oo(J')^\op$ is a map of posets. Vice versa, if $(<f,\geq f)\colon J \to \Oo(J')^\op$ is a map of posets then for every $j\leq k$ in $J$ we have $[f(j),+\infty)\subseteq [f(k),+\infty)$ and so $f(k)\leq f(j)$.
\end{proof}

\begin{prop}\label{graphs}
  The map 
  \begin{align*}
 \gr \colon \pos(J,\Oo(J')^\op)^\op&\to \Oo(J \times J')\\
   U & \mapsto \gr_U  
  \end{align*}
is an isomorphism of posets. 
\end{prop}

\begin{proof}
Let $U_1\leq U_2$ in the partial order on $\pos(J,\Oo(J')^\op)^\op$, and let $(j,j')\in \gr_{U_2}$. Then $j'\in U_{2;k}\subseteq U_{1,j}$ and so $(j,j')\in \gr_{U_1}$. This means $\gr_{U_1}\leq \gr_{U_2}$ in $\Oo(J \times J')$. Next, for any morphism of posets $U\colon J\to \Oo(J)^\op$ we have that $(j,j')\in \Gamma_{U+1}$ if and only if $j'\in (U+1)_j$
  Pick and upper set $\tilde{U}$ of $J \times J'$ and set, for $j \in J$ $$U_{\tilde{U}}(j)=\{ j' \in J'; \quad  (j,j') \in \tilde{U} \}.$$
  The subset $U_{\tilde{U}}(j)\subseteq J'$ is an upper set. Indeed, if $j'\in U_{\tilde{U}}(j)$ and $k'\geq j'$ in $J'$, then $(j,k')\geq (j,j')$ in the product order on $J\times J'$ and so $(j,k')\in \tilde{U}$ as $\tilde{U}$ is an upper set. Next, if $j\leq k$ in $J$ and $j'\in U_{\tilde{U}}(j)$, then $(k,j')\in \tilde{U}$ as $\tilde{U}$ is an upper set, and so $j'\in U_{\tilde{U}}(k)$. This shows that $U_{\tilde{U}}(j)\subseteq U_{\tilde{U}}(k)$, and so $U_{\tilde{U}}$ is a map of posets from $J$ to $\Oo(J')^\op$. Moreover, if $\tilde{U}_1\leq \tilde{U}_2$ in $\Oo(J\times J')$ then $U_{\tilde{U}_1}\leq U_{\tilde{U}_2}$ in $\pos(J,\Oo(J')^\op)^\op$. Therefore we have a map
   \begin{align*}
\gamma\colon   \Oo(J \times J') &\to \pos(J,\Oo(J')^\op)^\op \\
  \tilde{U} & \mapsto U_{\tilde{U}}.
  \end{align*}
  which is straightforward to see to be the inverse of $\gr$. 
\end{proof}

\subsection{Perversities}

Assume $J$ and $J'$ are $\Z$-posets. Then $\Oo(J)^\op$ is a $\Z$-poset with the generator $1$ acting as $U\mapsto U-1$. The action of $\Z$ by conjugation on $\pos(J,\Oo(J')^\op)$ is therefore given by $(U\dotplus 1)_j=U_{j-1}-1$. To see that $U\dotplus 1$ is still a morphism of posets from $J$ to $\Oo(J')$, let $j\leq k$ in $J$. Then $(U\dotplus 1)_j=U_{j-1}-1\subseteq U_{k-1}-1=(U\dotplus 1)_k$. 
The conjugation action, however, does not define a structure of $\Z$-poset on $\pos(J,\Oo(J')^\op)$, as it is generally not true that $U\leq U\dotplus1$. This condition is indeed equivalent to $U_{j}\subseteq (U\dotplus 1)_{j}$, i.e., to $U_{j}+1\subseteq U_{j-1}$ for any $j$ in $J$, and this is not necessarily satisfied by a morphism of posets $U\colon J\to \Oo(J')^\op$. 
%On the other hand, a morphism of posets $U\colon J\to \Oo(J')^\op$ surely satisfies $U_{j-1}\subseteq U_{j}$. This motivates the following.
\begin{defn}
A \emph{$(J,J')$-perversity} is a map of posets $U\colon J\to \Oo(J')^\op$ such that 
\[
U_{j}+1\subseteq U_{j-1}%\subseteq U_{j}
\]
for any $j$ in $J$. The set $\mathrm{Perv}(J,J')$ of all $(J,J')$-perversities inherits a poset structure from the inclusion in $\pos(J,\Oo(J')^\op)^\op$. It is a $\Z$-poset with $1$ acting as $U\mapsto U\dotplus(-1)$.
\end{defn}
\begin{rem}
Notice that the shift action on $(J,J')$-perversities is given by  $U\mapsto U\dotplus(-1)$. Namely, by construction the $\Z$-action $U\mapsto U\dotplus 1$ is monotone on the set of perversities with the order induced by the inclusion in $\pos(J,\Oo(J')^\op)$, and the order on $\mathrm{Perv}(J,J')$ is the opposite one. The reason for considering this order is, clearly, Proposition \ref{graphs}.
\end{rem}

We have seen in Lemma \ref{decreasing-gives-upper-set} that a function $f\colon J\to J'$ defines a morphism of posets $(\geq f)\colon J\to \Oo(J')^\op$ if and only if $f$ is a morphism of posets from $J$ to $J'^\op$. When this happens, $(\geq f)$ is a $(J,J')$-perversity if and only if $f(j-1)\leq f(j)+1$, for every $j\in J$. In the particular case $J=\Z$, 
assuming as usual that  $J'$ is a $\Z$-poset, we can define a new function $p_f\colon \Z\to J'$ as $p_f(n)=f(n)+n$. Then the condition $f(n)\leq f(n+1)+1$ translates into 
$p_f(n)\leq p_f(n+1)$, while the condition that $f\colon \Z\to J'^\op$ is a morphism of posets, i.e., $f(n+1)\leq f(n)$ translates to $p_f(n+1)\leq p_f(n)+1$.  As $f\mapsto p_f$ is a bijection of the set of maps from $J$ to $J'$ into itself, we see that the functions $f\colon \Z\to J'$ defining perversities correspond bijectively to the set of functions $p\colon \Z \to J'$ such that
$p(n)\leq p(n+1)\leq p(n)+1$, 
for every $n\in \Z$. This motivates the following (see \cite{bbd}).

\begin{defn}
A $(\Z,\Z)$-perversity $U$ is said to be defined by a function if there exists $f\colon \Z\to \Z$ such that $U=(\geq f)$. The subset of $(\Z,\Z)$-perversities defined by functions is denoted by $\mathrm{Perv}^\circ(\Z,\Z)$.
\end{defn}
\begin{lem}\label{lemma.perv0}
The subset $\mathrm{Perv}^\circ(\Z,\Z)$ is a $\Z$-sub-poset of $\mathrm{Perv}(\Z,\Z)$. It is isomorphic via $f\mapsto (\geq f)$ with the $\Z$-poset of nonincreasing functions $f\colon \Z\to \Z$ such that $f(n-1)\leq f(n)+1$ with the $\Z$-action given by $(f\dotplus 1)(n)=f(n+1)+1$.
\end{lem}
\begin{proof}
As a function $f\colon \Z\to \Z$ is uniquely determined by the collection of upper sets $[f(n),+\infty)$ the set $\mathrm{Perv}^\circ(J,J')$ bijectivley corresponds to the set of those functions $f\colon \Z\to \Z$ such that $(\geq f)$ is a $(\Z,\Z)$-perversity. As noticed above, these are precisely nonincreasing functions from $\Z$ to itself such that  $f(n-1)\leq f(n)+1$. This bijection is an isomorphism of posets, as $(\geq f_1)\leq (\geq f_2)$ if and only if $f_1\leq f_2$ (in the standard poset structure on the set of maps from $\Z$ to the poset $\Z$). Finally, $((\geq f)\dotplus (-1))_n=(\geq f)_{n+1}+1=[f(n+1)+1,+\infty)=(\geq (f\dotplus 1))_n$, for any $n\in \Z$.
\end{proof}

\begin{defn}\label{defperv}
A \emph{perversity function} (on $\Z$) is a function $p\colon \Z\to \Z$ such that
\[
p(n)\leq p(n+1)\leq p(n)+1,
\]
for every $n\in \Z$. It is called a strict perversity function if 
\[
p(n+2)-1\leq p(n)\leq p(n+1)\leq p(n)+1,
\]
for every $n\in \Z$. The set $\mathrm{perv}_\Z$ of perversity functions is a poset with the partial order induced by the inclusion $\mathrm{perv}_\Z\subseteq \pos(\Z,\Z)$. It is a $\Z$-poset with the action $(p\dotplus1)(n)=p(n+1)$. 
\end{defn}
\begin{rem}
A perversity function (on $\Z$) can be equivalently defined as a function $p\colon \Z\to \Z$ such that
\[
0\leq p(n)-p(m)\leq n-m
\]
for any $n-m\geq 0$ and the strictness condition translates to the additional condition $p(n)-p(m)< n-m$ for $n-m\geq 2$.

Basic examples of perversity functions are the zero perversity $p(n)\equiv 0$ and the identity perversity $p(n)\equiv n$. Another classical example is the \emph{middle perversity} $p(n)\equiv \lfloor n/2 \rfloor$. In particular, both the zero perversity and the middle perveristy are examples of strict perversities.
\end{rem}

\begin{rem}\label{other-action}
In addition to the $\Z$-action $p\mapsto p\dotplus 1$, the poset  $\mathrm{perv}_\Z$ carries also another natural $\Z$-action making it a $\Z$-poset; namely, the action $(p+1)(n)=p(n)+1$. We will come back to this towards the end of this section.
\end{rem}

\begin{rem}
Notice that the $\Z$-action on perversity functions is a monotone action since, by definition of perversity function, we have $p(n)\leq p(n+1)$ and this precisely means $p(n)\leq (p\dotplus1)(n)$.
\end{rem}
\begin{lem}
For every $p\colon \Z \to \Z$, let $f_p\colon \Z\to \Z$ be the map defined by $f_p(n)= p(n)-n$. Then $p\mapsto (\geq f_p)$ is an isomorphism of $\Z$-posets between $\mathrm{perv}_\Z$ and $\mathrm{Perv}^\circ(\Z,\Z)$
\end{lem}
\begin{proof}
By Lemma \ref{lemma.perv0} we only need to show that $p\mapsto f_p$ is a monotone $\Z$-equivariant bijection between  $\mathrm{perv}_\Z$ and the set of functions $f\colon \Z\to \Z$ such that $f(n)\leq f(n-1)\leq f(n)+1$. That it is a bijection is immediate: the inverse map is $f\mapsto p_f$, where $p_f(n)=f(n)+n$. To see that it is an isomorphism of posets, notice that $f_{p_1}\leq f_{p_2}$ if and only if $p_1(n)-n\leq p_2(n)-n$ for every $n\in \Z$, and so if and only if $p_1(n)\leq p_2(n)$ for every $n\in \Z$. Finally, $f_{p\dotplus1}(n)=(p\dotplus 1)(n)-n=p(n+1)-n=f_p(n+1)+1=(f_p\dotplus 1)(n)$.
\end{proof}

\subsubsection{Perversities as slicings of the lattice $\Z\times \Z$}

\begin{defn}
An upper set $U$ in $\Oo(J\times J')$ is called a \emph{kinky upperset} if $U\leq U+_\mathrm{ne}1$, where ${\mathrm{ne}}$ is the ``northwestern'' action of $\Z$ on $J\times J'$ given by
$(j,j')+_{\mathrm{ne}}1=(j-1,j'+1)$. We denote by $\mathrm{Kink}(J\times J')$ the poset of kinky uppersets of $J\times J'$, with the poset structure induced by the inclusion in $\Oo(J\times J')$. it is a $\Z$-poset with the ``northwestern'' action. We denote by $\mathrm{Kink}^\circ(J\times J')$ the $\Z$-sub-poset of nontrivial kinky uppersets of $J\times J'$, where 
the trivial upperstes are $\emptyset$ and $J\times J'$.
\end{defn}

\begin{lem}
 The map $\Gamma$ from Proposition \ref{graphs} induces an isomorphism of $\Z$-posets
  \[
 \gr \colon \mathrm{Perv}(J,J')\to \mathrm{Kink}(J\times J').
  \]
\end{lem}
\begin{proof}
Let $U\colon J\to \Oo(J')^\op$ be a perversity, and let $(j,j')\in \Gamma_U$. Then $j'\in U_{j}$ and so, by definition of perversity, $j'+1\in U_{j-1}$. Therefore $(j-1,j'+1)\in \Gamma_U$, i.e., $\Gamma_U\leq \Gamma_U+_{\mathrm{ne}} 1$. Vice versa, if $\tilde{U}$ is a kinky upperset in $J\times J'$, let $U_{\tilde{U}}$ the preimage in $\pos(J,\Oo(J')^\op)^\op$ of $\tilde{U}$ via $\Gamma$ (see the proof of Proposition \ref{graphs}). Then for any $j\in J$ and any $j'\in U_{\tilde{U};j}$ we have $j'+1\in U_{\tilde{U};j-1}$ and so $U_{\tilde{U};j}+1\subseteq U_{\tilde{U};j-1}$. So the isomorphism of posets $\pos(J,\Oo(J')^\op)^\op\to \Oo(J \times J')$ restricts to an isomorphism of posets $\mathrm{Perv}(J,J')\to \mathrm{Kink}(J\times J')$. To see that $ \gr \colon \mathrm{Perv}(J,J')\to \mathrm{Kink}(J\times J')$ is also $\Z$-equivariant, notice that, for every perversity $U$ we have $(j,j')\in \Gamma_{U\dotplus (-1)}$ if and only if $j'\in U_{j+1}+1$. i.e., if and only if $(j+1,j'-1)\in \Gamma_U$. This latter condition is equivalent to $(j,j')\in \Gamma_U+_{\mathrm{nw}}1$, so we find $\Gamma_{U\dotplus (-1)}=\Gamma_U+_{\mathrm{nw}}1$.
\end{proof}

\begin{lem}
The map $\Gamma$ from Proposition \ref{graphs} induces an isomorphism of $\Z$-posets
  \[
 \gr \colon \mathrm{Perv}^\circ(\Z,\Z)\to \mathrm{Kink}^\circ(\Z\times \Z).
  \]
\end{lem}
\begin{proof}
{Let $U$ be a kinky upper set that is not in the image of $\Gamma$. We want to show that $U=\emptyset$ or $U=\Z\times \Z$. To do this, we notice that}
a kinky upperset $U$ is in the image of $\Gamma$ if and only if $U$ is of the form $(\geq f)$ for a suitable function $f\colon \Z\to \Z$, {and that} this is possible if and only if $U_n\neq \emptyset,\Z$ for every $n\in \Z$.  As $U$ is kinky, if $U_{n_0}=\emptyset$ for some $n_0$, then $U_{n_0+1}+1\subseteq U_{n_0}=\emptyset$, and so $U_{n_0+1}=\emptyset$. Inductively, this gives $U_n=\emptyset$ for every $n\geq n_0$.  On the other hand, since a kinky upperset is an upperset, if there exists a nonempty $U_n$ with $n<n_0$, then there exist an element $(n,m)$ in $U$ and so, since $(n,m)\leq (n_0,m)$, also $ (n_0,m)\in U$. But then $m\in U_{n_0}$, which is impossible. So also the $U_n$ with $n<n_0$ are empty and therefore $U=\emptyset$. Similarly, if $U_{n_0}=\Z$ for some $n_0$, then $U_n=\Z$ for every $n>n_0$ as $U$ is an upperset, while the kinkiness condition $U_{n}+1\subseteq U_{n-1}$ implies that also $U_{n_0-1}=\Z$ and so, inductively, that all $U_n$ with $n<n_0$ are the whole of $\Z$. That is, $U=\Z\times\Z$ in this case. 
\end{proof}

\begin{lem}\label{equiv}
The isomorphism of $\Z$-modules $\varphi\colon \Z^2\to \Z^2$ given by $\varphi\colon (n,n')\mapsto (n+n',n')$ induces an isomorphism of $\Z$-posets 
\[
\varphi\colon \mathrm{Kink}(\Z\times \Z)\xrightarrow{\sim} \Oo(\Z\times \Z),
\]
where the $\Z$-action on $\Oo(\Z\times \Z)$ is the one induced by the ``northern'' $\Z$-action on $\Z^2$, namely, $(n,n')+1=(n,n'+1)$. In particular $\varphi$ induces an isomorphism of $\Z$-posets $\mathrm{Kink}^{\circ}(\Z\times \Z)\xrightarrow{\sim} \Oo(\Z\times \Z)\setminus\{\emptyset,\Z\times\Z\}$.
\end{lem}
\begin{proof}
A subset $U$ of $\Z\times \Z$ is an upper set (in the product order) if and only if $U+K\subseteq U$, where $K$ is the $\Z$-cone spanned by $\left(\begin{smallmatrix}1\\0\end{smallmatrix}\right)$ and $\left(\begin{smallmatrix}0\\1\end{smallmatrix}\right)$, while $U$ is a kinky upperset if and only if $U+K^\mathrm{kink}\subseteq U$, where $K^\mathrm{kink}$ is the $\Z$-cone spanned by $\left(\begin{smallmatrix}1\\0\end{smallmatrix}\right)$ and $\left(\begin{smallmatrix}-1\\1\end{smallmatrix}\right)$. Morever the $\Z$ action on the uppersets is generated by $U\mapsto U+\left(\begin{smallmatrix}0\\1\end{smallmatrix}\right)$ and the the $\Z$ action on the kinky uppersets is generated by $U\mapsto U+\left(\begin{smallmatrix}-1\\1\end{smallmatrix}\right)$, where in both cases the sum on the right hand side is the sum in $\Z^2$. As $\varphi \colon \Z^2\to \Z^2$ is an isomorphism of $\Z$-modules with $\varphi(K^\mathrm{kink})=K$ and $\varphi\left(\begin{smallmatrix}-1\\1\end{smallmatrix}\right)=\left(\begin{smallmatrix}0\\1\end{smallmatrix}\right)$, the statement follows ($\varphi$ is manifestly inclusion preserving).
\end{proof}

\begin{cor}\label{corperv}
We have an isomorphism of $\Z$-posets
\[
\mathrm{perv}_\Z\xrightarrow{\sim} \Oo(\Z\times \Z)\setminus\{\emptyset,\Z\times\Z\}
\]
mapping a perversity function $p$ to the image via the isomorphism $\varphi\colon (n,n')\mapsto (n+n',n')$ of the set $\{(n,n')\in \Z\times \Z\text{ such that } n'\geq p(n)-n\}$.
\end{cor}

\begin{rem}
The zero perversity function $p(n)\equiv 0$ corresponds to the upper set $\{(n,n')\in \Z\times \Z\text{ such that } n\geq 0\}$; the identity perversity function $p(n)\equiv n$ corresponds to the upper set $\{(n,n')\in \Z\times \Z\text{ such that } n'\geq 0\}$.
\end{rem}

\begin{rem}\label{missing}
The two ``missing'' upper sets from $\Oo(\Z\times \Z)\setminus\{\emptyset,\Z\times\Z\}$ can be recovered by adding to the set $\mathrm{perv}_\Z$ of perversity functions the two ``constant infinite perversities'', i.e., the function $p_{+\infty}\colon \Z\to \Z\cup\{-\infty,+\infty\}$ defined by $p_{+\infty}(n)=+\infty$ for every $n\in\Z$ and the function  $p_{-\infty}\colon \Z\to \Z\cup\{-\infty,+\infty\}$ defined by $p_{-\infty}(n)=-\infty$ for every $n\in\Z$. The extended set
\[
\widehat{\mathrm{perv}}_\Z =\mathrm{perv}_\Z\cup\{p_{-\infty},p_{+\infty}\}
\]
is naturally a $\Z$-poset with $p_{+\infty}$ and $p_{-\infty}$ as maximum and minimum element, respectively (so they are in particular $\Z$-fixed points), and the inclusion of $\mathrm{perv}_\Z$ into $\widehat{\mathrm{perv}}_\Z$ is a morphism of $\Z$-posets. Moreover, the isomorphism of $\Z$-posets $\mathrm{perv}_\Z\xrightarrow{\sim} \Oo(\Z\times \Z)\setminus\{\emptyset,\Z\times\Z\}$ from Corollary \ref{corperv} extends to an isomorphism of $\Z$-posets
\[
\widehat{\mathrm{perv}}_\Z\xrightarrow{\sim} \Oo(\Z\times \Z).
\]
\end{rem}
As well as the $\Z$-action $p\mapsto p\dotplus 1$, also the action $p\mapsto p+1$ from Remark \ref{other-action} extends to a $\Z$-action of $\widehat{\mathrm{perv}}_\Z$ (again, $p_{+\infty}$ and $p_{-\infty}$ will be fixed points). The isomorphism of posets $\widehat{\mathrm{perv}}_\Z\xrightarrow{\sim} \Oo(\Z\times \Z)$ will then transfer this $\Z$-action to $\Oo(\Z\times \Z)$ inducing on $\Oo(\Z\times \Z)$ a $\Z$-action such that $\widehat{\mathrm{perv}}_\Z\xrightarrow{\sim} \Oo(\Z\times \Z)$ is an isomorphism of posets with the $\Z$-action $p\mapsto p+1$ on the left hand side. More precisely, we have the following result.
\begin{prop}
\label{prop-perv}
We have an isomorphism of posets
\[
\widehat{\mathrm{perv}}_\Z\xrightarrow{\sim} \Oo(\Z\times \Z)
\]
mapping a perversity function $p$ to the image via the isomorphism $\varphi\colon (n,n')\mapsto (n+n',n')$ of the set $S_p=\{(n,n')\in \Z\times \Z\text{ such that } n'\geq p(n)-n\}$, and mapping the ``infinite perversity functions'' $p_{-\infty}$ and $p_{+\infty}$ to $\Z\times \Z$ and to $\emptyset$, respectively. Moreover, this is an isomorphism of $\Z$-posets, where the $\Z$-action on the left is given by $(p+1)(n)=p(n)+1$ and the $\Z$-action on the right is the one induced by the ``northeastern'' $\Z$-action on $\Z^2$, namely, $(n,n')+1=(n+1,n'+1)$. 
\end{prop}
\begin{proof}
After Corollary \ref{corperv} and Remark \ref{missing}, the only thing left to prove is the $\Z$-equivariancy of the isomorphism, i.e., that we have
\[
\varphi(S_{p+1})=\varphi(S_p)+\left(\begin{smallmatrix}1\\1\end{smallmatrix}\right).
\]
As $\varphi$ is an isomorphism of $\Z$-modules, this is equivalent to
\[
S_{p+1}=S_p+\left(\begin{smallmatrix}0\\1\end{smallmatrix}\right),
\]
i.e., to the condition $(n,n')\in S_{p+1}$ if and only if $(n,n'-1)\in S_{p}$, which is obvious.
\end{proof}
Taking the opposite of the complement (or, equivalently, the complement of the opposite) gives an isomorphism of posets
\begin{align*}
\Oo(\Z\times \Z)&\xrightarrow{\sim} \Oo(\Z\times \Z)^{\mathrm{op}}\\
U&\mapsto (\Z\times \Z)\setminus (-U),
\end{align*}
where $-U=\{(-n,-n')\text{ with } (n,n')\in U\}$. This isomorphisms changes the northeastern action in its opposite, i.e., the ``southwestern'' action  $(n,n')+1=(n-1,n'-1)$, so Proposition \ref{prop-perv} immediatley gives 
\begin{cor}
\label{cor-perv2}
We have an isomorphism of posets
\[
\widehat{\mathrm{perv}}_\Z^{\mathrm{op}}\xrightarrow{\sim} \Oo(\Z\times \Z)
\]
mapping a perversity function $p$ to the complement of the image via the isomorphism $\psi\colon (n,n')\mapsto (-n-n',-n')$ of the set $S_p=\{(n,n')\in \Z\times \Z\text{ such that } n'\geq p(n)-n\}$, and mapping the ``infinite perversity functions'' $p_{-\infty}$ and $p_{+\infty}$ to $\emptyset$ and to $\Z\times \Z$, respectively. Moreover, this is an isomorphism of $\Z$-posets, where the $\Z$-action on the left is given by $(p+^{\mathrm{op}}1)(n)=p(n)-1$ and the $\Z$-action on the right is the one induced by the ``northeastern'' $\Z$-action on $\Z^2$, namely, $(n,n')+1=(n+1,n'+1)$. 
\end{cor}

\subsection{Slicing the heart}\label{sec:heart}
Let $\mathscr{D}$ be a stable $\infty$-category { and let $\tee$ be a bounded t-structure on $\mathscr{D}$, i.e., equivalently, the datum of} a Bridgeland $\Z$-slicing of $\mathscr{D}$. Let $\heartsuit_{\mathfrak{t}}$ denote the heart of $\mathfrak{t}$. Then an abelian $\Z$-slicing of $\heartsuit_{\mathfrak{t}}$ is the datum of an extension $\tilde{\tee}$ of $\tee$ to a Bridgeland $\Z\times_{\mathrm{lex}}\hat{\Z}$-slicing on $\mathscr{D}$, where $\hat{\Z}$ denotes the $\Z$-poset consisting of $\Z$ endowed with the trivial $\Z$-action, and the morphism $\Oo(\Z)\to \Oo(\Z\times_{\mathrm{lex}}\hat{\Z})$ is induced by the projection on the first factor $\Z\times_{\mathrm{lex}}\hat{\Z}\to \Z$; see \cite[Section 5]{fosco}. We will denote by $\heartsuit_{\tee;\phi}$ the $\phi$-th slice of the heart of $\tee$. In other words,
\[
\heartsuit_{\tee;\phi}=\mathscr{D}_{\tilde{\tee};(0,\phi)}.
\]
Notice that we have
\[
\heartsuit_{\tee;\phi}[n]=\mathscr{D}_{\tilde{\tee};(n,\phi)},
\]
where $[n]$ denotes the ``shift by $n$'' functor on $\mathscr{D}$.

\begin{exmp}
Via the obvious inclusion of $\Z$-posets $\{0,1\}\hookrightarrow\hat{\Z}$, any torsion pair $(\heartsuit_{\tee;0},\heartsuit_{\tee;1})$ on $\heartsuit_{\mathfrak{t}}$ defines an abelian $\Z$-slicing on $\heartsuit_{\mathfrak{t}}$.
\end{exmp}
\begin{defn}\label{defgrad}
Let $\mathfrak{t}$ be a bounded t-structure on $\mathscr{D}$. An abelian $\mathbb{Z}$-slicing $\hat{\tee}$ on $\heartsuit_{\mathfrak{t}}$ is called:
\begin{itemize}
\item \emph{perverse} (or \emph{weak grading})
 if $\heartsuit_{\tee;\phi}\orth\heartsuit_{\tee;\psi}[n]$ for $\phi>\psi+n$; 
\item \emph{grading %filtration
} (or \emph{radical}) if $\heartsuit_{\tee;\phi}\orth\heartsuit_{\tee;\psi}[n]$ for $\phi>\psi+n$ and %$\heartsuit_{\tee;\phi }\orth \heartsuit_{\tee;\psi}[n]$
 for $\phi=\psi+n$ with $n\geq 2$; 
\item \emph{gluable} if $\heartsuit_{\tee;\phi}\orth \heartsuit_{\tee;\psi}[n]$ %and $\heartsuit_{\tee;\phi}\orth \heartsuit_{\tee;\psi}[n+1]$ 
for $\phi>\psi$ and $n>0$;
%\item \emph{mixed %filtration
%} if $\heartsuit_{\tee;\phi }\orth\heartsuit_{\tee;\psi}[n]$ for $ \phi > \psi -n$.  
\end{itemize} 
\end{defn}
\begin{rem}\label{special-gluable}
The definition of gluable abelian $\Z$-slicing of $\heartsuit_\tee$ is the specialization of Definition \ref{gluable} to $J_1=\Z$ and $J_2=\hat{\Z}$.
\end{rem}
\begin{rem}
Historically, grading filtrations first appeared in \cite{ekh} under the name of `radical filtrations'.%, while mixed filtrations are covered in the last section of \cite{kos}.
\end{rem}
\begin{exmp}
Let  $(\heartsuit_{\tee;0},\heartsuit_{\tee;1})$ be a torsion pair on $\heartsuit_{\mathfrak{t}}$. Then $(\heartsuit_{\tee;0},\heartsuit_{\tee;1})$, seen as an abelian $\Z$-slicing, is grading. Namely, as $\heartsuit_{\tee;\phi}[n]=0$ for $\phi \notin\{ 0,1\}$, the only nontrivial orthogonality conditions to be checked are:
\begin{itemize}
\item[-] $\heartsuit_{\tee;0}\orth\heartsuit_{\tee;0}[n]$ for $n<0$;
\item[-] $\heartsuit_{\tee;0}\orth\heartsuit_{\tee;1}[n]$ for $n<-1$;
\item[-] $\heartsuit_{\tee;1}\orth\heartsuit_{\tee;0}[n]$ for $n<1$;
\item[-] $\heartsuit_{\tee;1}\orth\heartsuit_{\tee;1}[n]$ for $n<0$.
\end{itemize}
These all follows from the orthogonality relation $\heartsuit_{\tee}\orth \heartsuit_\tee[n]$ for $n< 0$, except for $\heartsuit_{\tee;1}\orth\heartsuit_{\tee;0}$ which is true by definition of torsion pair.
\end{exmp}

%\begin{lem}\label{mix-to-glue}
%Let $\mathfrak{t}$ be a bounded t-structure on $\mathscr{D}$, and let $\hat{\tee}$ be an abelian $\mathbb{Z}$-slicing on $\heartsuit_{\mathfrak{t}}$. If $\tilde{\tee}$ is mixed, then $\tilde{\tee}$ is gluable.
%\end{lem}
%\begin{proof}
%Let $\phi$ and $\psi$ be in $\Z$ with $\phi>\psi$, and let $n>0$. Then $\phi>\psi-n$ and so, by definition of mixed slicing,  $\heartsuit_{\tee;\phi }\orth\heartsuit_{\tee;\psi}[n]$. Therefore, $\tilde{\tee}$ is gluable.
%\end{proof}
\begin{prop}\label{glue-to-grad}
Let $\mathfrak{t}$ be a bounded t-structure on $\mathscr{D}$, and let $\tilde{\tee}$ be an abelian $\mathbb{Z}$-slicing on $\heartsuit_{\mathfrak{t}}$. If $\tilde{\tee}$ is gluable, then $\tilde{\tee}$ is grading.
\end{prop}
\begin{proof}
Let $\phi$ and $\psi$ be in $\Z$ with $\phi>\psi+n$.
The orthogonality condition $\heartsuit_{\tee;\phi}\orth\heartsuit_{\tee;\psi}[n]$ is trivially satisfied if $n<0$, so let us assume $n\geq 0$. If $n=0$, then $\phi>\psi$ and the orthogonaliy condition $\heartsuit_{\tee;\phi}\orth\heartsuit_{\tee;\psi}$ is satisfied by definition of slicing. Finally, if $n>0$, then we have $\phi>\psi$ and $n>0$, so $\heartsuit_{\tee;\phi}\orth\heartsuit_{\tee;\psi}[n]$ by definition of gluable slicing. %This shows that $\tilde{\tee}$ is perverse. 
If $\phi=\psi+n$ with $n\geq 2$, then in particular $\phi>\psi$ and $n>0$, so again $\heartsuit_{\tee;\phi}\orth\heartsuit_{\tee;\psi}[n]$. 
\end{proof}

\begin{prop}\label{grad}
Let $\mathfrak{t}$ be a bounded t-structure on $\mathscr{D}$, and let $\tilde{\tee}$ be a perverse abelian $\mathbb{Z}$-slicing on $\heartsuit_{\mathfrak{t}}$. Then, $\tilde{\tee}$ is $g_p$-compatible, where 
\begin{align*}
g_p\colon \mathbb{Z} \times_{\mathrm{lex}} \hat{\mathbb{Z}}&\to \mathbb{Z} \times_{\mathrm{lex}} \hat{\mathbb{Z}}\\
(n,\phi)&\mapsto(n+p(\phi),-p(\phi)),
\end{align*}
for every  strict perversity function $p\colon \Z\to \Z$. If $\tilde{\tee}$ is grading, then $\tilde{\tee}$ is $g_p$-compatible for every  perversity function $p$.
\end{prop}
\begin{proof}
The map $g_p$ is $\Z$-equivariant, as the action on the second factor is the trivial one.  Let $(n,\phi)$ and $(m,\psi)$ in $\Z\times_{\mathrm{lex}} \hat{\Z}$ with $(n,\phi)\leq (m,\psi)$ such that $g_p(n,\phi)>g_p(m,\psi)$ in $\Z\times_{\mathrm{lex}} \hat{\Z}$. By definition of $g_p$ this means that we have
\[
(n+p(\phi),-p(\phi))>(m+p(\psi),-p(\psi))
\]
in $\Z\times_{\mathrm{lex}} \Z$, i.e., that $n+p(\phi)>m+p(\psi)$ or that $n+p(\phi)=m+p(\psi)$ and $p(\psi)>p(\phi)$. Similarly, the condition $(n,\phi)\leq (m,\psi)$ means that either $n<m$ or $n=m$ and $\phi\leq \psi$. By considering all possibilities, and taking into account that a perversity function is nondecreasing, one sees that there is actually a single case to deal with:
%\begin{itemize}
%\item 
$p(\phi)-p(\psi)>m-n$, with $m>n$ and $\phi>\psi$.
%\item $n=m$ and $p(\phi)>p(\psi)$ with $\phi\leq \psi$.
%\item $n+p(\phi)=m+p(\psi)$ with $n<m$ and $p(\psi)>p(\phi)$;
%\item $n+p(\phi)=m+p(\psi)$ and $p(\psi)>p(\phi)$ and $n=m$ and $\phi\leq \psi$.
%\end{itemize}
If $p$ is a perversity function, if $\phi> \psi$ then
\[
0\leq p(\phi)-p(\psi)\leq \phi-\psi,
\]
so we have
\[
\phi\geq  \psi+ p(\phi)-p(\psi)>\psi+m-n.
\]
Since $\tilde{\tee}$ is perverse, this implies $\heartsuit_{\tee;\phi}\orth\heartsuit_{\tee;\psi}[m-n]$, and so $\heartsuit_{\tee;\phi}[n]\orth \heartsuit_{\tee;\psi}[m]$, i.e.,
\[
\mathscr{D}_{\tilde{\tee};(n,\phi)}\orth\mathscr{D}_{\tilde{\tee};(m,\psi)}.
\]
Concerning the orthogonality of $\mathscr{D}_{\tilde{\tee};(n,\phi)}$ and $\mathscr{D}_{\tilde{\tee};(m,\psi)}[1]$, notice that
from $\phi\geq \psi+ p(\phi)-p(\psi)>\psi+m-n$ it follows that either $\phi>\psi+m-n+1$ or $\phi=\psi+m-n+1\geq \psi+2$.  If $\phi>\psi+m-n+1$, then reasoning as above we find $\mathscr{D}_{\tilde{\tee};(n,\phi)}\orth\mathscr{D}_{\tilde{\tee};(m,\psi)}[1]$. If $\phi=\psi+m-n+1\geq \psi+2$,
there are two cases to be considered. In the first case, $\tilde{\tee}$ is perverse and the perversity function $p$ is strict. In the second case, we allow $p$ to be any perversity, but we put a restriction on $\tilde{\tee}$, which we require to be grading.

In the first case, as $p$ is strict and $\phi\geq \psi+2$, we have $p(\phi)-p(\psi)< \phi-\psi$, and so again $\phi>\psi+m-n+1$.
In the second case, as $m-n+1\geq 2$ and $\tilde{\tee}$ is grading, we have $\heartsuit_{\tee;\phi}\orth\heartsuit_{\tee;\psi}[m-n+1]$, i.e., again
\[
\mathscr{D}_{\tilde{\tee};(n,\phi)}\orth\mathscr{D}_{\tilde{\tee};(m,\psi)}[1].
\]
\end{proof}

\begin{rem}
The proof of Proposition \ref{grad} makes it clear the meaning of the otherwise obscure condition in the definition of grading abelian $\Z$-slicing: the condition on a shift by at least 2 in the definition of strict perversity is traded for an  orthogonality condition in the slicing. 
\end{rem}

%\begin{prop}\label{grad2}
%Let $\mathfrak{t}$ be a bounded t-structure on $\mathscr{D}$, let $\tilde{\tee}$ be an abelian $\mathbb{Z}$-slicing on $\heartsuit_{\mathfrak{t}}$,  let $p\colon \Z\to \Z$ be a perversity function, and let $g_p\colon \mathbb{Z} \ltimes \hat{\mathbb{Z}}\to \mathbb{Z} \ltimes \hat{\mathbb{Z}}$ be the $\Z$-equivariant map given by $$g_p(n,\phi)=(n+p(\phi),-p(\phi)),$$
%where $\Z$ acts diagonally both on the source and on the target. With some abuse of nototation, denote by ${g_p}_!\tee$ the bounded $t$-structure on $\mathscr{D}$ induced by the $\Z\times \hat{\Z}$-slicing ${g_p}_!\tilde{\tee}$ via the projection on the first factor $\Z\times \hat{\Z}\to \Z$.
%If $\tilde{\tee}$ is mixed, then  then the abelian $\mathbb{Z}$-slicing ${g_p}_!\tilde{\tee}$ on $\heartsuit_{{g_p}_!\tee}$ is split. 
%\end{prop}

\begin{rem}\label{alpha}
By taking $p$ to be the identity perversity, $\mathrm{id}\colon \Z\to \Z$, we see that if $\tilde{\tee}$ is grading then $\tilde{\tee}$ is $\alpha$-compatible, where $\alpha\colon \mathbb{Z} \times_{\mathrm{lex}} \hat{\mathbb{Z}}\to \mathbb{Z} \times_{\mathrm{lex}} \hat{\mathbb{Z}}$ is the $\Z$-equivariant map given by
\[
\alpha(n,m)=(n+m,-m).
\]
\end{rem}

The proof of the following lemma is straightforward.
\begin{lem}\label{from-alpha-to-e}
Let $\beta\colon \mathbb{Z} \times_{\mathrm{lex}} \hat{\mathbb{Z}}\to \mathbb{Z} \times_{\mathrm{lex}} {\mathbb{Z}}$ be the map defined by
\[
\beta(n,m)=(n,n+m).
\]
Then $\beta$ is an isomorphism of $\Z$-tosets. Moreover the diagram
\[
\xymatrix{
\mathbb{Z} \times_{\mathrm{lex}} \hat{\mathbb{Z}}\ar[r]^{\alpha}\ar[d]_{\beta}& \mathbb{Z} \times_{\mathrm{lex}} \hat{\mathbb{Z}}\ar[d]^{\beta}\\
\mathbb{Z} \times_{\mathrm{lex}} {\mathbb{Z}}\ar[r]^{e}& \mathbb{Z} \times_{\mathrm{lex}} {\mathbb{Z}}
}
\]
commutes,
where $e\colon \Z\times \Z \to \Z\times \Z$ is the exchange map $e(n,m)=(m,n)$ from Lemma \ref{exchange} and $\alpha\colon \mathbb{Z} \times_{\mathrm{lex}} \hat{\mathbb{Z}}\to \mathbb{Z} \times_{\mathrm{lex}} \hat{\mathbb{Z}}$ is the map $\alpha(n,m)=(n+m,-m)$ from Remark \ref{alpha}.
\end{lem}

\begin{cor}
Let $\mathfrak{t}$ be a bounded t-structure on $\mathscr{D}$, and let let $\tilde{\tee}$ be an abelian $\mathbb{Z}$-slicing on $\heartsuit_{\mathfrak{t}}$. Then, in the notation of Lemma \ref{from-alpha-to-e}, $\tilde{\tee}$ is $\alpha$-compatible if and only if $\beta_*\tee$ is gluable.
\end{cor}
\begin{proof}
By Lemma \ref{from-alpha-to-e} and Remark \ref{avanti-e-indietro}, $\tilde{\tee}$ is $\alpha$-compatible if and only if $\tilde{\tee}$ is $(\beta\circ\alpha)$-compatible, which happens if and only if $\tilde{\tee}$ is $(e\circ\beta)$-compatible. As $\beta$ is an isomorphism of $\Z$-tosets, we can write $\tee=(\beta^{-1})_*\beta_*\tee$, and so, by Remark \ref{avanti-e-indietro} again, $\tilde{\tee}$ is $(e\circ\beta)$-compatible if and only if $\beta_*\tee$ is $e$-compatible. By Remark \ref{special-gluable}, this is equivalent to saying that $\beta_*\tee$ is gluable. 
\end{proof}
From Proposition \ref{glue-to-grad} and Remark \ref{alpha} we immediately get the following
\begin{cor}
Let $\mathfrak{t}$ be a bounded t-structure on $\mathscr{D}$, and let let $\tilde{\tee}$ be an abelian $\mathbb{Z}$-slicing on $\heartsuit_{\mathfrak{t}}$. If $\tilde{\tee}$ is grading, then $\beta_*\tee$ is gluable. In particular, if $\tilde{\tee}$ is grading, then also $\beta_*\tee$ is grading.
\end{cor}

\begin{prop}\label{grad2}
Let $\mathfrak{t}$ be a bounded t-structure on $\mathscr{D}$, and let let $\tilde{\tee}$ be a grading abelian $\mathbb{Z}$-slicing on $\heartsuit_{\mathfrak{t}}$. For every perversity $p$, 
let $\gamma_p\colon \Z\times_{\mathrm{lex}}\hat{\Z}\to \Z$ the $\Z$-equivariant morphism given by 
\[
\gamma_p\colon (n,\phi)\mapsto n+p(\phi).
\]
Then
\begin{itemize}
\item $\tilde{\tee}$ is $\gamma_p$-compatible;
\item $ (\gamma_p)_!\tilde{\tee}$ is a bounded $t$-structure on $\mathscr{D}$;
\item the map
\begin{align*}
\Psi\colon \mathrm{perv}_\Z^{\mathrm{op}}&\to \ts(\mathscr{D})\\
p&\mapsto (\gamma_p)_!\tilde{\tee},
\end{align*}
is a morphism of $\Z$-posets, where on the left we have the $\Z$-action $(p+1)(n)=p(n)+1$ and on the right the $\Z$-action given by the shift in $\mathscr{D}$;
\item 
$\Psi$ uniquely extends to a morphism of $\Z$-posets
\[
\Psi\colon \widehat{\mathrm{perv}}_\Z^{\mathrm{op}}\to \ts(\mathscr{D})
\]
preserving maxima and minima.
\end{itemize}
\end{prop}
\begin{proof}
Let $g_p\colon \mathbb{Z} \times_{\mathrm{lex}} \hat{\mathbb{Z}}\to \mathbb{Z} \times_{\mathrm{lex}} \hat{\mathbb{Z}}$ the map defined in Proposition \ref{grad}
As $\tilde{\tee}$ is grading, by Proposition \ref{grad}, $\tilde{\tee}$ is $g_p$-compatible. The projection on the first factor, $\pi_1\colon  \mathbb{Z} \times_{\mathrm{lex}} \hat{\mathbb{Z}}\to \Z$ is a morphism of $\Z$-tosets, so by Remark \ref{everything-compatible} every $\mathbb{Z} \times_{\mathrm{lex}} \hat{\mathbb{Z}}$-slicing is $\pi_1$-compatible. In particular, $(g_p)_!\tee$ is $\pi_1$-compatible. Therefore, by Proposition \ref{functoriality}, $\tee$ is $(\pi_1\circ g_p)$-compatible. As $\pi_1\circ g_p=\gamma_p$, this precisely says that $\tee$ is $\gamma_p$-compatible. We therefore have a Bridgeland $\Z$-slicing, i.e., a bounded $t$-structure, $(\gamma_p)_!\tee$ on $\mathscr{D}$. The map $\Psi$ is monotone and $\Z$-equivariant. Indeed, the $t$-structure $\Psi_p$ is defined by the upper category
\[
\mathscr{D}_{(\gamma_{p})_!\tilde{\tee};\geq 0}=\langle \mathscr{D}_{\tee;(n,\phi)}\rangle_{(n,\phi)\in \gamma_p^{-1}([0,+\infty))}.
\]
If $p_1\geq^{\mathrm{op}} p_2$, then  $p_1\leq p_2$ and so $n+p_1(\phi)\geq 0$ implies $n+p_2(\phi)\geq 0$, and so $\gamma_{p_1}^{-1}([0,+\infty))\subseteq \gamma_{p_2}^{-1}([0,+\infty))$. This gives $\mathscr{D}_{(\gamma_{p_1})_!\tilde{\tee};\geq 0}\subseteq \mathscr{D}_{(\gamma_{p_2})_!\tilde{\tee};\geq 0}$, i.e.,
\[
\mathscr{D}_{(\gamma_{p_1})_!\tilde{\tee}}\geq \mathscr{D}_{(\gamma_{p_2})_!\tilde{\tee}}
\]
in the parial order on  $\ts(\mathscr{D})$. Similarly, $n+(p+^{\mathrm{op}}1)(\phi)\geq 0$ if and only if $n+p(\phi)\geq 1$ and so
\[
\mathscr{D}_{(\gamma_{p+1})_!\tilde{\tee};\geq 0}=\mathscr{D}_{(\gamma_{p})_!\tilde{\tee};\geq 1}=\mathscr{D}_{(\gamma_{p})_!\tilde{\tee};\geq 0}[1].
\]
Finally, $\Psi$ trivially (and uniquely) extends to $\widehat{\mathrm{perv}}_\Z^{\mathrm{op}}$ preserving maxima and minima.
\end{proof}
Recalling Corollary \ref{cor-perv2} and Proposition \ref{glue-to-grad} we finally get the result we were aiming to.
\begin{thm}\label{main-thm}
Let $\mathfrak{t}$ be a bounded t-structure on $\mathscr{D}$, and let let $\tilde{\tee}$ be a gluing abelian $\mathbb{Z}$-slicing on $\heartsuit_{\mathfrak{t}}$. Then $\tilde{\tee}$ explicitly induces a natural morphism of $\Z$-posets
\[
\Psi\colon\Oo(\Z\times\Z)\to \ts(\mathscr{D})
\]
such that for every proper upper set in $\Z\times \Z$, the corresponding $t$-structure on $\mathscr{D}$ is bounded. 
\end{thm}
\begin{rem}By Proposition \ref{grad2}, one sees that
Theorem \ref{main-thm} is actually true under the weaker assumption that $\tilde{\tee}$ is grading. Moreover, by restricting to the $\Z$-sub-poset of $\Oo(\Z\times\Z)$ consisting of the image in $\Oo(\Z\times\Z)$ of strict perversities, Theorem \ref{main-thm} holds under the even weaker assumption that $\tilde{\tee}$ is perverse. We preferred to state it under the stronger assumption of a gluable $\tilde{\tee}$ to make its use in the examples below more immediate. See however Subsection \ref{perverse-coherent} for an geometrically interesting example involving a perverse abelian slicing.
\end{rem}

\begin{rem}\label{main-rem}
The morphism $\Psi$ can be thought of as a $(\Z\times \Z)$-slicing of $\mathscr{D}$, but one has to keep in mind that the poset $\Z\times\Z$ indexing the slices (and so the cohomologies) is now not totally ordered. This is a possibly subtle point, so let us spend a few more words on it. An abelian $\mathbb{Z}$-slicing $\tilde{\tee}$ of $\heartsuit_\tee$ is by definition a $(\Z\times_{\mathrm{lex}}\hat{\Z})$-slicing, and by Lemma \ref{from-alpha-to-e}, this is equivalently a $(\Z\times_{\mathrm{lex}}{\Z})$-slicing. So going from an abelian slicing of the heart to a  $(\Z\times_{\mathrm{lex}}{\Z})$-slicing of $\mathscr{D}$ is a trivial step. What is nontrivial is going from an abelian slicing of the heart to a $(\Z\times{\Z})$-slicing of $\mathscr{D}$, where now the poset structure on $\Z\times \Z$ is given by the product order and not by the lexicographic order. And indeed this can generally not be done for an 
arbitrary abelian slicing of the heart, and here is where the property of the abelian slicing to be grading comes in. Finally, to emphasize once more how going from a   $(\Z\times_{\mathrm{lex}}{\Z})$-slicing to a $(\Z\times{\Z})$-slicing is a nontrivial step, consider how there are many more upper sets  in $\Z\times{\Z}$ than in $\Z\times_{\mathrm{lex}}{\Z}$.
\end{rem}
\begin{rem}
Describing the bounded $t$-structure on $\mathscr{D}$ associated by Theorem \ref{main-thm} to a proper upper set $U$ of $\Z\times \Z$ is a bit involved, but it is a completely explicit procedure. To begin with, recall that a bounded $t$-structure is completely determined by its heart, so we only need to give a description of the heart $\heartsuit_U$ associated with $U$. To do this, notice that the perveristy function associated to $U$ is
\[
p_U(n)=n+\min\{n'\in \Z\text{ such that } (n,n')\notin \psi^{-1}(U)\} 
\]
where $\psi^{-1}(n,n')=(-n+n',-n')$. The heart $\heartsuit_U$ is then the extension closed subcategory of $\mathscr{D}$ generated by the slices $\mathscr{D}_{(-p_U(n),n)}$ of $\tilde{\tee}$.
\end{rem}
\begin{exmp}
If $U=\{(n,n')\text{ such that } n'\geq 1\}$, then $\psi(U)=\{(n,n')\text{ such}$ $\text{that } n'\leq -1\}$ and so $p_U(n)=n$. Therefore, $\heartsuit_U=\langle \mathscr{D}_{(-n,n)}\rangle_{n\in \Z}$ in this case. If $U=\{(n,n')\text{ such that } n\geq 1\}$, then $\psi(U)=\{(n,n')\text{ such that } n'\leq -n-1\}$ and so $p_U(n)=0$. Therefore, $\heartsuit_U=\langle \mathscr{D}_{(0,n)}\rangle_{n\in \Z}$ in this case.
\end{exmp}
A more explict description of the \emph{perverse hearts} of $\mathscr{D}$ is as follows.
\begin{thm}\label{perverse-heart}
Let $\mathfrak{t}$ be a bounded t-structure on $\mathscr{D}$, and let let $\tilde{\tee}$ be a gluing abelian $\mathbb{Z}$-slicing on $\heartsuit_{\mathfrak{t}}$. Let $U$ be an upper set of $\Z\times \Z$ and let $p$ be the corresponding perversity. Then the preverse heart $\heartsuit_p=\heartsuit_U$ of $\mathscr{D}$ is the full subcategory of $\mathscr{D}$ on those objects $X$ such that 
\[
H_\tee^{n}(X)[-n]\in \langle \heartsuit_{\tee;n'}\rangle_{n'\in p^{-1}(-n)}
\]
for every $n\in \Z$, where $H_\tee^{n}(X)$ is the $n$-th cohomology object of $X$ in the $t$-structure $\tee$ and $\{\heartsuit_{\tee;n'}\}_{n'\in \Z}$ are the slices of the heart $\heartsuit_\tee$ of $\tee$ for the abelian $\Z$-slicing $\tilde{\tee}$.
 \end{thm}
 \begin{proof}
 Denote by $\tee_p$ the $\tee$-structure on $\mathscr{D}$ associated with the perversity function $p$. Then the lower subcategory $\mathscr{D}_{\tee_p;<0}$ and the upper subcategory $\mathscr{D}_{\tee_p;\geq 0}$ of $\mathscr{D}$ are defined, by Proposition \ref{grad}, as 
 \[
 \mathscr{D}_{\tee_p;<0}=\langle \mathscr{D}_{\tilde{\tee};(n,n')}\rangle_{n+p(n')<0}; \qquad \mathscr{D}_{\tee_p;\geq 0}=\langle \mathscr{D}_{\tilde{\tee};(n,n')}\rangle_{n+p(n')\geq 0}.
 \]
 These can be equivalently described as
\begin{align*}
 \mathscr{D}_{\tee_p;<0}&=\{X\in \mathscr{D}\text { such that } H_{\tilde{\tee}}^{(n,n')}(X)=\mathbf{0}\text{ for } n+p(n')\geq 0\};\\
   \mathscr{D}_{\tee_p;\geq 0}&=\{X\in \mathscr{D}\text { such that } H_{\tilde{\tee}}^{(n,n')}(X)=\mathbf{0}\text{ for } n+p(n')< 0\},
\end{align*}
 see, e.g., \cite[Remark 4.27]{fosco}. Therefore
 \[
 \heartsuit_p=\{X\in \mathscr{D}\text{ such that } H_{\tilde{\tee}}^{(n,n')}(X)=\mathbf{0}\text{ for } n+p(n')\neq 0\}.
 \]
 Equivalently, this means that
  \[
 \heartsuit_p=\{X\in \mathscr{D}\text{ such that } H_{\tee}^{n}(X)\in \langle \heartsuit_{\tee;n'}[n]\rangle_{p(n')=-n}\}.
 \]
\end{proof}

\begin{exmp}
Let $k\in \Z$ and let $\chi_{[k,+\infty)}\colon \Z\to \{0,1\}$ the characteristic function of the interval $[k,+\infty)$. Seen as a function from $\Z$ to $\Z$, the function $\chi_{[k,+\infty)}$is a perversity function of a very special kind: it is a perversity function taking exactly two values. Moreover, it is easy to see that --up to an additive constant-- perversity functions taking exactly two values are precisely characteristic functions of upper intervals in $\Z$. We have
\[
\chi_{[k,+\infty)}^{-1}(-n)=\begin{cases}
\emptyset &\text{if }n\neq -1,0\\
\\
[k,+\infty)&\text{if }n= -1\\
\\
(-\infty,k)&\text{if }n =0.
\end{cases}
\]
Therefore the perverse heart $\heartsuit_{\chi_{[k,+\infty)}^{-1}}$ of $\mathscr{D}$ is the full subcategory of $\mathscr{D}$ on those objects $X$ such that 
\[
\begin{cases}
H_\tee^{n}(X[1])= \mathbf{0} \qquad \text{ if }n\neq 0,1 \\
\\
H_\tee^{0}(X[1])= \langle \heartsuit_{\tee;n'}\rangle_{n'\in [k,+\infty)}\\
\\
H_\tee^{1}(X[1])= \langle \heartsuit_{\tee;n'}[1]\rangle_{n'\in (-\infty,k)}
\end{cases}
\]
for every $n\in \Z$. In other words, $\heartsuit_{\chi_{[k,+\infty)}^{-1}}$ is (up to a shift by 1) the heart of the tilted $t$-structure obrained by tilting $\tee$ with the torsion theory on $\heartsuit_\tee$ given by $\mathcal{F}=\langle \heartsuit_{\tee;n'}\rangle_{n'\in (-\infty,k)}$ and $\mathcal{T}=\langle \heartsuit_{\tee;n'}\rangle_{n'\in [k,+\infty)}$.
\end{exmp}

%\begin{rem}
%  In {\color{red}ref: BBD} the authors consider a process consisting in gluing $t$-structures on a pair of triangulated categories to obtain a $t$-structure on a third one. In their setting, the three categories are linked by an abstraction of the Grothendieck's six-functors formalism. However, the latter is more or less equivalent to the datum of a semiorthogonal decomposition (i.e., a $\{0,1 \}$-slicing, where $\{0,1 \}$ is the toset with two elements and trivial $\Z$-action) on the third category so that the input of the two (bounded) $t$-structures gives rise to a gluable $\{0,1\} \lex \Z$-slicing $\tee$. Then the $\Z \lex \{0,1 \}$-slicing $e_! \tee$ is just the new glued $t$-structure together with the natural torsion pair on its heart. This explains our terminology and links the language of the present paper to the classical theory of $t$-structures. 
%\end{rem}
%
%{\color{red} l'osservazione che segue dovr\`a trovare la sua collocazione
%
%\begin{rem}
%To motivate the next section, consider the following isomorphism fo $\Z$-tosets:
%  \begin{align*}
%     \Z \lex \Z &\to \Z \lex \Z_{\textnormal{free}} \\
%   (n,m) & \mapsto (n,m-n)  
%  \end{align*}
%\end{rem}
%
%... continua ...
%}

\section{A zoo of examples}

The  upshot of Theorem \ref{main-thm} is that out of a gluing abelian $\Z$-slicing on the heart of a bounded $t$-structure on a stable $\infty$-category $\mathscr{D}$
we explicitly get a natural morphism of $\Z$-posets
\[
\Psi\colon\Oo(\Z\times\Z)\to \ts(\mathscr{D}),
\]
which, on the subset of perversities, acts as
\begin{align*}
\Psi\colon \mathrm{perv}_\Z^{\mathrm{op}}&\to \ts(\mathscr{D})\\
p&\mapsto (\gamma_p)_!\tilde{\tee},
\end{align*}
see Proposition \ref{grad}. We can build this way whole new classes of `perverse' $t$-structures on $\mathscr{D}$. In this section we present a few examples, reinterpreting and rediscovering a few classical results from the literature on perverse $t$-structures within the unifying framework provided by the construction presented in the main section of the article.

\subsection{An example from algebra and one from geometry}

\subsubsection{The nonstandard $t$-structure from Koszul duality}

 An instance of a nonstandard construction of a $t$-structure lies within the theory of Koszul duality. Namely, under suitable finiteness and semisimplicity assumptions, if $B$ and $\check{B}$ are Koszul dual algebras, then there is an equivalence of stable $\infty$-categories $\mathscr{D}^b(\mathrm{gmod}(\check{B}))\xrightarrow{\sim}\mathscr{D}^b(\mathrm{gmod}({B}))$, where $\mathrm{gmod}(\check{B})$ and $\mathrm{gmod}({B})$ are the categories of finitely generated graded modules over $\check{B}$ and $B$, respectively; see \cite{kosz}. This equivalence, however, does not preserve the standard $t$-structures, and the standard $t$-structure on $\mathscr{D}^b(\mathrm{gmod}(\check{B}))$ is mapped to a nonstandard `diagonal' $t$-structure $\mathscr{D}^b(\mathrm{gmod}({B}))$. This diagonal $t$-structure is actually an example of perverse $t$-structure deriving from a gluable abelian slicing on the standard heart of $\mathscr{D}^b(\mathrm{gmod}(\check{B}))$. Namely, the standard heart $\heartsuit_\tee$ of $\mathscr{D}^b(\mathrm{gmod}({B}))$ is the abelian category of finitely generated $\Z$-graded $B$-modules. For $\phi \in \Z$, denote by $\heartsuit_{\tee;\phi}$ the full subcategory of $\heartsuit_{\tee}$ consisting of modules concentrated in degree $\phi$. Clearly, this defines an abelian $\Z$-slicing $\tilde{\tee}$ on $\heartsuit_{\tee}$. Following \cite{kosz} we have
  \[
  \textnormal{Ext}_{B}^n(M_{\phi},M_{\psi})=0
  \]
for $n>\psi - \phi$, for any modules $M_\phi\in\heartsuit_{\tee;\phi}$ and $M_\psi\in\heartsuit_{\tee;\psi}$ . In particular, if $\phi>\psi$ and $n>0$ we have  $\heartsuit_{\tee;\phi}\orth \heartsuit_{\tee;\psi}[n]$, and so
$\tilde{\tee}$ is a gluable slicing. Therefore, for every perversity $p$  we have a bounded perverse $t$-structure $(\gamma_p)_!\tilde{\tee}$ on $\mathscr{D}^b(\mathrm{gmod}({B}))$. By choosing $p$ to be the identity perversity $\mathrm{id}\colon \Z\to \Z$ we get a distinguished perverse $t$-structure $(\gamma_{\mathrm{id}})_!\tilde{\tee}$ on $\mathscr{D}^b(\mathrm{gmod}({B}))$. This is precisely the `diagonal' $t$-structure %arising in Koszul duality, see 
considered in \cite{kosz}.

\subsubsection{Strictly perverse coherent sheaves}\label{perverse-coherent}

 Let $X$ be a smooth projective variety over $\mathbb{C}$ \footnote{This assumption can be greatly weakened, see \cite{bezr}.} and let $\mathscr{D} = \mathscr{D}^b(\textrm{Coh}(X))$, the (bounded) derived category of coherent sheaves on $X$, endowed with its canonical heart $\heartsuit = \textrm{Coh}(X)$. Then there is an abelian $\{0, \cdots , \dim X \}$-slicing on $\heartsuit$ given by defining $\heartsuit_i$ as the full subcategory of $\textrm{Coh}(X)$ on coherent sheaves with support of pure codimension $i$, for each $0 \leq i \leq \dim X$. 
 %The abelian Postnikov towers simply amount to the classical filtrations by codimension of support. 
 We can see $\{\heartsuit_i\}_{0\leq i\leq \dim X}$ as an abelian $\Z$-slicing of $\heartsuit$ via the obvious $\Z$-poset embedding $\{0, \cdots , \dim X \} \subseteq \hat{\Z}$. By Serre duality and Grothendieck vanishing, if $\mathscr{E}$ and $\mathscr{F}$ are coherent sheaves on $X$ then
 \[
 \mathrm{Ext}^n(\mathscr{E},\mathscr{F})=0\qquad \text{for } n<\dim \mathrm{supp}(\mathscr{F})-\dim\mathrm{supp}(\mathscr{E}).
 \]
  Equivalently, this can be rewritten as $\mathscr{D}(\mathscr{E},\mathscr{F}[n])=0$ for $n<i-j$, for any $\mathscr{E}$ in $\heartsuit_i$ and $\mathscr{F}$in $\heartsuit_j$, i.e., 
 $\heartsuit_{i}\orth\heartsuit_{j}[n]$ for $i>j+n$. In other words, $\{\heartsuit_i\}_{0\leq i\leq \dim Xn}$ is a perverse abelian $\Z$-slicing on $\textrm{Coh}(X)$. Therefore, by Theorem \ref{main-thm} and Remark \ref{main-rem}, with any strict perversity function $p$ is associated a perverse $t$-structure on  $\mathscr{D}$, i.e., one gets a lattice of $t$-structures on  $\mathscr{D}$ parametrized by strict perversities. The heart of these perverse $t$-structures are the perverse coherent sheaves constructed in  \cite{bezr}, while the lattice of $t$-structures on $\mathscr{D}$ associated with the abelian $\Z$-slicing $\{\heartsuit_i\}_{0\leq i\leq \dim X}$ is the `distributive lattice of $t$-structures' from \cite{bondper}.

\subsection{Gluable slicings from baric structures}\label{baric}

  The notion of a bounded Bridgeland $\hat{\Z}$-slicing is not new: it already appears in literature under other names. Namely, it is no more than an infinite version of a semiorthogonal decomposition in the sense of \cite{semiort}, or a `\textit{baric structure}' as defined in \cite{baric}. These are a rich source of gluable abelian $\Z$-slicings. Namely, given a baric structure $\{\mathscr{D}_{n}\}_{n\in \mathbb{Z}}$ on a stable $\infty$-category $\mathscr{D}$ together with the datum of a bounded $t$-structure on each of the stable subcategories $\mathscr{D}_{n}$, we can look at this as the datum of a $\hat{\Z} \lex \Z$-slicing $\hat{\tee}$ on $\mathscr{D}$. If $\hat{\tee}$ is gluable, then $e_!\hat{\tee}$ is a gluable ${\Z} \lex \hat{\Z}$-slicing of $\mathscr{D}$, by Lemma \ref{incolla2}. By the results in Section \ref{sec:heart}, $e_!\hat{\tee}$ is equivalently an abelian slicing on the heart $\heartsuit_{{\tee}}$ of the bounded $t$-structure ${\tee}$ on $\mathscr{D}$ defined by the composition
  \[
 \Oo(\Z)\to \Oo(\Z\times_{\mathrm{lex}}\hat{\Z})\xrightarrow{e_!\hat{\tee}} \ts(\mathscr{D}).
  \]
 and so it is a gluable abelian slicing. {Moreover, and remarkably, the gluability of $\hat{\tee}$ can be easily explicited. Namely, spelling out Definition \ref{gluable}, we see that $\hat{\tee}$ is gluable if and only if  $\mathscr{D}_{i;\geq 0}\orth \mathscr{D}_{j;0}$ for any $i<j$. {As  $\mathscr{D}_{i;\geq 0}$ is generated by the subcategories $\mathscr{D}_{i}^\heartsuit[k]$ for $k\geq 0$, this is equivalent to $\mathscr{D}_{i}^\heartsuit\orth\mathscr{D}_{j}^\heartsuit[n]$ whenever $i<j$ and $n\leq 0$. This generalizes the gluability condition from Example \ref{example:bbd}.} 
 }

\subsubsection{Gluability and the Beilinson-Soul\'e conjecture}\label{motives}

%  In this example we relate the gluability condition with the Beilinson-Soul\'e conjecture from motivic topology. We fix a field $k$ and, just for this example, we stick to a more traditional $1$-categorical setup. 

The existence of motives, which is still an open question in general, was conjectured by Grothendieck in order to build a universal Weil cohomology theory for schemes: the `motivic cohomology'. Following this input, Deligne observed that it could be easier to construct first a triangulated category (the `mixed' motives) which should play the role of the derived category of motives, and later recover the abelian category of motives as the heart of a bounded $t$-structure on mixed motives. Finally, Voedvodskij succeded in constructing a triangulated category of mixed rational motives over a characteristic zero field $\mathbbm{k}$.
%\footnote{Or, more generally, over a perfect field $k$ which admits resolution of singularities, see \cite{friedlander-voedvosky-bivariant}}
This triangulated category contains, for any $n \in \Z$, a `Tate object' $\mathbb{Q}(n)$ that represents the $n$-th motivic cohomology functor. Let $\mathscr{D}\textnormal{TM}_\mathbbm{k}$ be the category of mixed rational Tate motives, i.e., by definition, the triangulated\footnote{As, up to our knowledge, a description of the stable $\infty$-category of mixed motives is not available in the literature, in this subsection we stick to the more traditional $1$-categorical setup of triangulated categories. The same consideration applies to the examples considered in the subsequent subsections.} subcategory of the Voedvodskij category of mixed motives over $k$ generated by the Tate objects.  Denote by $(\mathscr{D}\textnormal{TM}_\mathbbm{k})_{m}$ the triangulated subcategory of $\mathscr{D}\textnormal{TM}_\mathbbm{k}$ generated by $\mathbb{Q}(m)$. One has an isomorphisms of groups
  \[
   \mathscr{D}\textnormal{TM}_\mathbbm{k}(\mathbb{Q}(i),\mathbb{Q}(j)[n])=K_{2(j-i)-n}(\mathbbm{k})^{(j-i)}
   \]
   where $K_a(\mathbbm{k})$ is the $a$-th higher $K$-theory group of the point $\textnormal{Spec}(\mathbbm{k})$ and $K_a(\mathbbm{k})^{(b)}$ is the weight $b$ summand of $K_a(\mathbbm{k}) \otimes_{\mathbb{Z}} \mathbb{Q}$ with respect to the Adams action. %In other words, $K_a(k)^{(b)}$ is the $a$-th $b$-codimensional Bloch's higher Chow group of $\textnormal{Spec}(k)$ with rational coefficients \cite{???}.

  By dimensional reasons, the right hand side vanishes for $i > j$ and for $i = j$ with $n \neq 0$. In other words, the Tate objects form an infinite exceptional collection on $\mathscr{D}\textnormal{TM}_\mathbbm{k}$ which is clearly full by definition. By the general theory of semiorthogonal decomposition, this implies that the triangulated subcategories $(\mathscr{D}\textnormal{TM}_\mathbbm{k})_{m}$ with $m\in \Z$ are the slices of a baric structure on $\mathscr{D}\textnormal{TM}_\mathbbm{k}$ and that each of these slices is equivalent to $\mathscr{D}^b(\mathbb{Q}\text{-Vect})$,
 the bounded derived category  of finite-dimensional $\mathbb{Q}$-vector spaces, via an equivalence mapping $\mathbb{Q}(m)$ to $\mathbb{Q}$. %(which is equivalent to the abelian category graded vector spaces with finite-dimensional homogeneous components). 
   Since $\mathscr{D}^b(\mathbb{Q}\text{-Vect})$ is a bounded derived category, it comes equipped with a canonical bounded $t$-structure. The equivalences $(\mathscr{D}\textnormal{TM}_\mathbbm{k})_{m}\simeq \mathscr{D}^b(\mathbb{Q}\text{-Vect})$ then endow each slice of the baric structure with a bounded $t$-structure, whose heart $(\mathscr{D}\textnormal{TM}_\mathbbm{k})_{m}^\heartsuit$ is the abelian category generated by $\mathbb{Q}(m)$. The datum of these canonical $t$-structures on the slices of the baric structure  $\{(\mathscr{D}\textnormal{TM}_\mathbbm{k})_{m}\}_{m\in \mathbb{Z}}$ defines a Bridgeland $\hat{\Z}\lex \Z$ slicing on $\mathscr{D}\textnormal{TM}_\mathbbm{k}$ which, by the result in Section \ref{baric}, is gluable if and only if  
\[
(\mathscr{D}\textnormal{TM}_\mathbbm{k})_{i}^\heartsuit\orth(\mathscr{D}\textnormal{TM}_\mathbbm{k})_{j}^\heartsuit[n]
\]
 whenever $i<j$ and $n\leq 0$. As $(\mathscr{D}\textnormal{TM}_\mathbbm{k})_{m}^\heartsuit$ is generated by $\mathbb{Q}(m)$, the gluability condition is equivalent to
 \[
   \mathscr{D}\textnormal{TM}_\mathbbm{k}(\mathbb{Q}(i),\mathbb{Q}(j)[n])=\mathbf{0}
    \]
 whenever $i<j$ and $n\leq 0$, and therefore to    
\[
K_{2(j-i)-n}(\mathbbm{k})^{(j-i)} = \mathbf{0}
\]  whenever  $i < j$ and $n \leq 0$. This is exactly the Beilinson-Soul\'e standard vanishing conjecture, which is known to hold, for instance, when $\mathbbm{k}$ is a number field due to Borel's computation of the ranks of K-theory groups in this case \cite{borel}. When the conjecture holds, by applying $e_!$ we get a Bridgeland $\Z \lex \hat{\Z}$-slicing on $\mathscr{D}\textnormal{TM}_\mathbbm{k}$ and thus in particular a bounded $t$-structure whose heart contains the desired unmixed Tate motives over $\mathbbm{k}$. In other words, we recover a well known but %highly 
nontrivial fact (see \cite{levine}) using an abstract and very general reasoning: assuming the Beilinson-Soul\'e conjecture is true, (Tate) motives exist. \\
   Moreover, following the reasoning recalled at the beginning of this Section, we also get a $t$-structure on $\mathscr{D}\textnormal{TM}_\mathbbm{k}$ for each perversity function on $\mathbb{Z}$. These are the `\textit{perverse motives}' appearing in \cite{permot}. 

\subsubsection{Three more examples}

There are a number of other constructions in literature which are a particular case of the one we presented here. Just to mention a few, in \cite{beil} Beilinson defines a notion of \textit{'filtered structure'} on a triangualted category. This is no more than a baric structure $\{ \mathscr{D}_{n} \}_{n \in \hat{\Z}}$ with some additional data and properties. Starting with a $t$-structure on $\mathscr{D}_0$, Beilinson rearranges it into a $t$-structure on $\mathscr{D}$. One can easily check that the axioms of a filtered structure guarantee that it defines a gluable $\hat{\Z} \lex \Z$-slicing and that the distinguished $t$-structure obtained by gluing coincides with Beilinson's new $t$-structure on $\mathscr{D}$. \\ 
\par

In \cite{macri}, Macr\'i starts with a finite `Ext exceptional' collection on a certain triangualted cateogory $\mathscr{D}$ and get a distinguished $t$-structure on $\mathscr{D}$. This construction actually goes along the exact lines sketched in Subsection \ref{motives}. Namely, when translated into the language of this note, a finite exceptional collection is just a baric structure with finitely many nonzero slices, all equivalent to the derived categogry of finite-dimensional vector spaces over some fixed field, and the condition of being `Ext exceptional' is identified with the gluability condition. \\
\par
Finally, a possibly more exotic instance is in \cite{lagra}. Here, starting with a suitable $\mathbb{R}$-slicing on the Fukaya category $\mathscr{D}_0$ of a symplectic manifold $M$, Hensel builds a $t$-structure on the Fukaya category $\mathscr{D}$ of $\mathbb{C} \times M$. This is done by embedding $\mathscr{D}_0$ into a triangulated category $\mathscr{D}$ as the zeroth slice of a baric structure. Lemma 7.1 from \cite{lagra} can then be reformulated as our gluability condition, and the distinguished $t$-structure obtained by gluing is seen to be the $t$-structure on the Lagrangian cobordism category exhibited by \cite{lagra}. \\

%\subsubsection{The Chinese Connection}
%{\color{red} qua metterei anche quella del cinese Yu Qiu non ancora pubblicata, ma devo prima parlare con lui e chiedere il permesso}.  

\newpage

\afterpage{\blankpage}
\clearpage 
\bibliographystyle{alpha}
\bibliography{biblio}

\begin{thebibliography}{{Hen}17}

\bibitem[AB10]{bezr}
Dmitry Arinkin and Roman Bezrukavnikov.
\newblock Perverse coherent sheaves.
\newblock {\em Mosc. Math. J.}, 10(1):3--29, 271, 2010.

\bibitem[AT11]{baric}
Pramod~N. Achar and David Treumann.
\newblock Baric structures on triangulated categories and coherent sheaves.
\newblock {\em Int. Math. Res. Not. IMRN}, (16):3688--3743, 2011.

\bibitem[BBD82]{bbd}
A.~A. Be{\u\i}linson, J.~Bernstein, and P.~Deligne.
\newblock Faisceaux pervers.
\newblock In {\em Analysis and topology on singular spaces, {I} ({L}uminy,
  1981)}, volume 100 of {\em Ast\'erisque}, pages 5--171. Soc. Math. France,
  Paris, 1982.

\bibitem[Bei]{beil}
A.~A. Beilinson.
\newblock On the derived category of perverse sheaves.

\bibitem[BGS96]{kosz}
Alexander Beilinson, Victor Ginzburg, and Wolfgang Soergel.
\newblock Koszul duality patterns in representation theory.
\newblock {\em J. Amer. Math. Soc.}, 9(2):473--527, 1996.

\bibitem[Bon13]{bondper}
A.~I. Bondal.
\newblock Operations on {$t$}-structures and perverse coherent sheaves.
\newblock {\em Izv. Ross. Akad. Nauk Ser. Mat.}, 77(4):5--30, 2013.

\bibitem[Bor74]{borel}
Armand Borel.
\newblock Stable real cohomology of arithmetic groups.
\newblock {\em Ann. Sci. \'Ecole Norm. Sup. (4)}, 7:235--272 (1975), 1974.

\bibitem[Bri07]{bridgeland}
Tom Bridgeland.
\newblock Stability conditions on triangulated categories.
\newblock {\em Ann. of Math. (2)}, 166(2):317--345, 2007.

\bibitem[CP10]{collins}
John Collins and Alexander Polishchuk.
\newblock Gluing stability conditions.
\newblock {\em Adv. Theor. Math. Phys.}, 14(2):563--607, 2010.

\bibitem[Eke86]{ekh}
Torsten Ekedahl.
\newblock {\em Diagonal complexes and {$F$}-gauge structures}.
\newblock Travaux en Cours. [Works in Progress]. Hermann, Paris, 1986.
\newblock With a French summary.

\bibitem[FLM15]{fosco}
D.~{Fiorenza}, F.~{Loregian}, and G.~{Marchetti}.
\newblock {Hearts and towers in stable infinity-categories}.
\newblock {\em ArXiv e-prints}, January 2015.

\bibitem[GKR04]{gkr}
A.~L. Gorodentsev, S.~A. Kuleshov, and A.~N. Rudakov.
\newblock {$t$}-stabilities and {$t$}-structures on triangulated categories.
\newblock {\em Izv. Ross. Akad. Nauk Ser. Mat.}, 68(4):117--150, 2004.

\bibitem[{Hen}17]{lagra}
F.~{Hensel}.
\newblock Stability conditions and lagrangian cobordisms.
\newblock {\em ArXiv e-prints}, December 2017.

\bibitem[Kuz14]{semiort}
Alexander Kuznetsov.
\newblock Semiorthogonal decompositions in algebraic geometry.
\newblock In {\em Proceedings of the {I}nternational {C}ongress of
  {M}athematicians---{S}eoul 2014. {V}ol. {II}}, pages 635--660. Kyung Moon Sa,
  Seoul, 2014.

\bibitem[Lev93]{levine}
Marc Levine.
\newblock Tate motives and the vanishing conjectures for algebraic
  {$K$}-theory.
\newblock In {\em Algebraic {$K$}-theory and algebraic topology ({L}ake
  {L}ouise, {AB}, 1991)}, volume 407 of {\em NATO Adv. Sci. Inst. Ser. C Math.
  Phys. Sci.}, pages 167--188. Kluwer Acad. Publ., Dordrecht, 1993.

\bibitem[Mac07]{macri}
Emanuele Macr\`\i.
\newblock Stability conditions on curves.
\newblock {\em Math. Res. Lett.}, 14(4):657--672, 2007.

\bibitem[SW18]{permot}
Wolfgang Soergel and Matthias Wendt.
\newblock Perverse motives and graded derived category {${\mathcal{O}}$}.
\newblock {\em J. Inst. Math. Jussieu}, 17(2):347--395, 2018.

\end{thebibliography}

\end{document}